\documentclass{amsart}
\usepackage[utf8]{inputenc}
\usepackage{amssymb}
\usepackage{mathrsfs} 
\usepackage{enumitem}
\usepackage{hyperref}
\usepackage[final]{showkeys} 

\input xy
\xyoption{all}

\theoremstyle{definition}
\newtheorem{mydef}{Definition}[section]
\newtheorem{lem}[mydef]{Lemma}
\newtheorem{thm}[mydef]{Theorem}

\newtheorem{cor}[mydef]{Corollary}

\newtheorem{question}[mydef]{Question}
\newtheorem{hypothesis}[mydef]{Hypothesis}

\newtheorem{defin}[mydef]{Definition}
\newtheorem{example}[mydef]{Example}
\newtheorem{remark}[mydef]{Remark}
\newtheorem{notation}[mydef]{Notation}
\newtheorem{fact}[mydef]{Fact}

\newcommand{\fct}[2]{{}^{#1}#2}

\newcommand{\bM}{\bar{M}}
\newcommand{\ba}{\bar{a}}
\newcommand{\bb}{\bar{b}}
\newcommand{\bc}{\bar{c}}
\newcommand{\bd}{\bar{d}}

\newcommand{\cf}[1]{\text{cf} (#1)}
\newcommand{\seq}[1]{\langle #1 \rangle}
\newcommand{\rest}{\upharpoonright}

\newcommand{\s}{\mathfrak{s}}
\newcommand{\gs}{\s}

\newcommand{\bs}{\text{bs}}

\newcommand{\leap}[1]{\le_{#1}}
\newcommand{\ltap}[1]{<_{#1}}

\newcommand{\geap}[1]{\ge_{#1}}

\newcommand{\lta}{\ltap{\K}}
\newcommand{\lea}{\leap{\K}}

\newcommand{\gea}{\geap{\K}}

\newcommand{\ltas}{\ltap{\s}}
\newcommand{\leas}{\leap{\s}}

\newcommand{\K}{\mathfrak{K}}

\newbox\noforkbox \newdimen\forklinewidth
\forklinewidth=0.3pt \setbox0\hbox{$\textstyle\smile$}
\setbox1\hbox to \wd0{\hfil\vrule width \forklinewidth depth-2pt
 height 10pt \hfil}
\wd1=0 cm \setbox\noforkbox\hbox{\lower 2pt\box1\lower
2pt\box0\relax}
\def\unionstick{\mathop{\copy\noforkbox}\limits}
\newcommand{\nf}{\unionstick}
\newcommand{\nfs}[4]{#2 \nf_{#1}^{#4} #3}

\def\1nf{\unionstick^{(1)}}

\def\2nf{\unionstick^{(2)}}
\def\3nf{\unionstick^{(3)}}

\newcommand{\tp}{\otp}

\newcommand{\Ss}{\oS}

\newcommand{\oSp}[1]{\mathscr{S}_{#1}}
\newcommand{\oS}{\oSp{}}
\newcommand{\oSs}{\oSp{\s}}
\newcommand{\oSsbs}{\oSs^{\bs}}

\newcommand{\Sbs}{\oSsbs}

\newcommand{\Ii}{\mathbb{I}}

\newcommand{\Ll}{\mathbb{L}}
\newcommand{\otp}{\mathbf{ortp}}

\newcommand{\goodp}{\text{good}^+}
\newcommand{\Eat}{E_{\text{at}}}

\newcommand{\PC}{\operatorname{PC}}
\newcommand{\WNF}{\operatorname{WNF}}
\newcommand{\LWNF}{\operatorname{LWNF}}
\newcommand{\VWNF}{\operatorname{VWNF}}
\newcommand{\RWNF}{\operatorname{RWNF}}
\newcommand{\NF}{\operatorname{NF}}

\newcommand{\F}{\mathcal{F}}
\newcommand{\LS}{\text{LS}}

\title{Abstract elementary classes stable in $\aleph_0$}
\date{\today}

\author {Saharon Shelah}
\email{shelah@math.huji.ac.il}
\urladdr{http://shelah.logic.at}
\address{Einstein Institute of Mathematics\\
Edmond J. Safra Campus, Givat Ram\\
The Hebrew University of Jerusalem\\
Jerusalem, 91904, Israel, and Department of Mathematics\\
 Hill Center - Busch Campus \\ 
 Rutgers, The State University of New Jersey \\
 110 Frelinghuysen Road \\
 Piscataway, NJ 08854-8019, USA}
\thanks{The first author would like to thank the Israel Science Foundation for partial support of this research (Grant No. 242/03).}
\thanks{Research partially supported by European Research Council grant 338821, and by National Science Foundation grant no: 136974. 1119 on Shelah's publication  list}

\author{Sebastien Vasey}
\email{sebv@math.harvard.edu}
\urladdr{http://math.harvard.edu/\textasciitilde sebv/}
\address{Department of Mathematical Sciences, Carnegie Mellon University, Pittsburgh, Pennsylvania, USA}
\address{Current address: Department of Mathematics \\ Harvard University \\ Cambridge, Massachusetts, USA}

\begin{document}
\begin{abstract}
  We study abstract elementary classes (AECs) that, in $\aleph_0$, have amalgamation, joint embedding, no maximal models and are stable (in terms of the number of orbital types). Assuming a locality property for types, we prove that such classes exhibit superstable-like behavior at $\aleph_0$. More precisely, there is a superlimit model of cardinality $\aleph_0$ and the class generated by this superlimit has a type-full good $\aleph_0$-frame (a local notion of nonforking independence) and a superlimit model of cardinality $\aleph_1$. We also give a supersimplicity condition under which the locality hypothesis follows from the rest.
\end{abstract}

\maketitle

\section{Introduction}
\subsection{Motivation}
In \cite{sh88} (a revised version of which appears as \cite[Chapter I]{shelahaecbook}, from which we cite), the first author introduced \emph{abstract elementary classes} (AECs): a semantic framework generalizing first-order model theory and also encompassing logics such as $\Ll_{\omega_1, \omega}$. He studied $\PC_{\aleph_0}$-representable AECs (roughly, AECs which are reducts of a class of models of a first-order theory omitting a countable set of types) and generalized and improved some of his earlier results on $\Ll_{\omega_1, \omega}$ \cite{sh87a, sh87b} and $\Ll_{\omega_1, \omega} (Q)$ \cite{sh48}.

For example, fix a $\PC_{\aleph_0}$-representable AEC $\K$ and assume that it is categorical in $\aleph_0$. Assuming $2^{\aleph_0} < 2^{\aleph_1}$ and $1 \le \Ii (\K, \aleph_1) < 2^{\aleph_1}$, the first author shows (without even assuming $\PC_{\aleph_0}$-representability) \cite[I.3.8]{shelahaecbook} that $\K$ has amalgamation in $\aleph_0$. Further, \cite[\S I.4, \S I.5]{shelahaecbook}, it has a lot of structure in $\aleph_0$ and assuming more set-theoretic assumptions as well as few models in $\aleph_2$, $\K$ has a superlimit model in $\aleph_1$ \cite[I.5.34, I.5.40]{shelahaecbook}. This means roughly (see Section \ref{prelim-sec}) that there is a saturated model in $\aleph_1$ and that the union of an increasing chain of type $\omega$ consisting of saturated models of cardinality $\aleph_1$ is saturated.

The reader can think of the existence of a superlimit in $\aleph_1$ as a step toward showing that the models of cardinality $\aleph_1$ behave in a ``superstable-like'' way. Indeed several recent works \cite{vandieren-chainsat-apal, vv-symmetry-transfer-afml, bv-sat-afml, gv-superstability-jsl} have connected superlimits with other definitions of superstability in AECs, including uniqueness of limit models and local character of orbital splitting.

Another notable consequence of the existence of a superlimit in $\aleph_1$ is that it implies that there is a model of cardinality $\aleph_2$. This ties back to a result of the first author: \cite[I.3.11]{shelahaecbook}: for a $\PC_{\aleph_0}$ AEC, categoricity in $\aleph_0$ and $\aleph_1$ implies the existence of a model in $\aleph_2$. The argument first establishes, using only categoricity in $\aleph_0$ and few models in $\aleph_1$, that there is a pair $M, N$ of models in $\K_{\aleph_1}$ such that $M \lta N$ and then uses, in essence, that (by $\aleph_1$-categoricity) these models are superlimits. In this context, the very strong hypotheses make it possible to avoid referring to any stability-theoretic notions. Still, in more complicated frameworks the existence of a superlimit model in $\aleph_1$ can be thought of as a key conceptual step toward proving existence of models in higher cardinality and more generally developing a stability theory cardinal by cardinal.

The arguments for the results from \cite[\S I.5]{shelahaecbook} discussed in the second paragraph of this introduction are complicated by the lack of $\aleph_0$-stability: one can only get that there are $\aleph_1$-many orbital types over countable models. The workaround there is to redefine the ordering (but not the class of models) to obtain a stable class, see \cite[I.5.29]{shelahaecbook}. If the AEC is ``nicely-presented'', e.g.\ a class of models of an $\Ll_{\omega_1,\omega}$-sentence or more generally a finitary AEC \cite{finitary-aec}, then this difficulty does not occur (see \cite{baldwin-larson-iterated}): $\aleph_0$-stability follows from few models in $\aleph_1$ and $2^{\aleph_0} < 2^{\aleph_1}$. One can also obtain $\aleph_0$-stability by starting with only countably-many models in $\aleph_1$ \cite[3.18]{almost-galois-stable}. Finally, it is worth noting that (assuming amalgamation and joint embedding in $\aleph_0$), $\aleph_0$-stability is upward absolute for $\PC_{\aleph_0}$-AECs \cite{lrsh1073}.

\subsection{Main results}

The bottom line is that $\aleph_0$-stability holds in several cases of interest. In fact, there are no known examples which (under $2^{\aleph_0} < 2^{\aleph_1}$) are categorical in $\aleph_0$, have few models in $\aleph_1$, and are \emph{not} $\aleph_0$-stable (see \cite[Question 3.15]{almost-galois-stable}). Thus in the present paper, we \emph{start} with stability in $\aleph_0$ (and often amalgamation and categoricity in $\aleph_0$). Our goal is to say as much as we can on the structure of the class, in particular to get superstable-like behavior in $\aleph_0$ and $\aleph_1$, \emph{without} assuming a non-ZFC hypothesis or $\Ii (\K, \aleph_1) < 2^{\aleph_1}$.

One of our first results (Theorem \ref{def-thm}) is that $\aleph_0$-stability (together with amalgamation and $\aleph_0$-categoricity) imply that the class $\K$ is already $\PC_{\aleph_0}$-representable. We also show that the assumption of categoricity in $\aleph_0$ is not really needed: without assuming it, one can find a superlimit in $\aleph_0$ and change to the class generated by that superlimit, which will be categorical in $\aleph_0$. In fact, we prove (Theorem \ref{brimmed-homog-equiv}) that one can characterize brimmed models (also called limit models in the literature) as those that are homogeneous for orbital types. This has as immediate consequence that the brimmed model of cardinality $\aleph_0$ is superlimit (Corollary \ref{sl-cor}). This last result sheds light on an argument of Lessmann \cite{lessmann-upward-transfer} and answers a question of Fred Drueck (see footnote 3 on \cite[p.~25]{drueckthesis}), who asked when this equality held. The argument works more generally assuming only density of amalgamation bases, as in \cite{shvi635}.

For the main result of this paper, we assume that orbital types over countable models are determined by their finite restrictions. The study of statements of the form ``orbital types are determined by their small restrictions'' was pioneered by Grossberg and VanDieren \cite{tamenessone}, who called this condition tameness. Hyttinen and Kesälä \cite[\S3]{finitary-aec} were the first to specifically study orbital types over finite sets and the condition that they determine orbital types over countable models. 

Following the first author's terminology \cite[0.1(2)]{sh932}, we call this last condition $(<\aleph_0, \aleph_0)$-locality (not to be confused with sequence locality \cite[0.1(1)]{sh932}, which is called locality in \cite[1.8]{non-locality} or \cite[11.4]{baldwinbook09}). This is known to hold for several classes of interest:

\begin{example}\label{locality-example} \
  \begin{enumerate}
  \item Let $\K$ be a finitary AEC (see \cite{finitary-aec}; this includes classes of models of $\Ll_{\omega_1,\omega}$-sentences) and assume that $\K$ is stable in $\aleph_0$ (finitary AECs have amalgamation and no maximal models by definition). By \cite[4.11]{finitary-aec}, $\K$ is $(<\aleph_0, \aleph_0)$-local. 
  \item Finitary AECs are not the only setup where $(<\aleph_0, \aleph_0)$-locality holds. For example, it is known for quasiminimal pregeometry classes (that may not be finitary \cite[Theorem 2]{kirby-note-axiom}), see \cite[5.2]{quasimin-five}, and more generally in the finite U-rank (FUR) classes of Hyttinen and Kangas \cite[2.17]{group-config-kangas-apal} (we thank Will Boney for pointing us to that reference). 
  \end{enumerate}
\end{example}

We prove the following:

\begin{thm}\label{test-thm}
  Let $\K$ be an AEC with $\LS (\K) = \aleph_0$ and countable vocabulary. Assume that $\K$ is categorical in $\aleph_0$, $\K$ is $(<\aleph_0, \aleph_0)$-local, $\K$ has amalgamation and no maximal models in $\aleph_0$ and $\K$ is stable in $\aleph_0$. Then:

  \begin{enumerate}
  \item (Theorem \ref{good-frame-thm}) There is a good $\aleph_0$-frame on $\K_{\aleph_0}$.
  \item (Corollary \ref{sl-aleph1-cor}) There is a superlimit model of cardinality $\aleph_1$.
  \end{enumerate}
\end{thm}

The good $\aleph_0$-frame (or the superlimit in $\aleph_1$) imply the nontrivial corollary that $\K$ has a model of cardinality $\aleph_2$ \cite[II.4.13]{shelahaecbook}. This consequence also follows from a theorem of the second author \cite[12.1]{stab-spec-v5} (which however does \emph{not} give a good $\aleph_0$-frame or a superlimit in $\aleph_1$). The conclusion that there is a superlimit model in $\aleph_1$ seems new, even for finitary AECs or FUR classes.

It is natural to ask whether the locality hypothesis in Theorem \ref{test-thm} is really needed\footnote{In fact, an earlier version of the present paper asserted that it could be derived from the other hypotheses but the argument contained a mistake.}. In fact we do not even know whether the existence of a good $\aleph_0$-frame implies $(<\aleph_0, \aleph_0)$-locality:

\begin{question}
  Let $\K$ be an AEC with $\LS (\K) = \aleph_0$. If there is a good $\aleph_0$-frame on $\K$, is $\K$ $(<\aleph_0, \aleph_0)$-local?
\end{question}

In Section \ref{supersimple-sec}, we give a partial answer: any AEC that is $\aleph_0$-stable, $\aleph_0$-categorical, and \emph{supersimple} (in a sense generalizing that of homogeneous model theory \cite{buechler-lessmann}) is $(<\aleph_0,\aleph_0)$-local. This generalizes the proof of \cite[5.2]{quasimin-five}, which shows that quasiminimal pregeometry classes are $(<\aleph_0, \aleph_0)$-local (see also \cite{quasimin-aec-v7-toappear}). Supersimple $\aleph_0$-stable AECs are also much more general than FUR classes.

\subsection{Notes}

This paper is organized as follows. Section \ref{prelim-sec} gives some background definitions and fixes the notation.  Section \ref{good-frame-sec} is a technical section on good frames (possibly on uncountable models) which sets up the machinery to prove the main theorem (more precisely, to prove a strong symmetry property for nonforking in good frames). Section \ref{sl-sec} works with countable models and shows that $\aleph_0$-stability implies the existence of a superlimit in $\aleph_0$. Section \ref{building-aleph0-frame} builds the good $\aleph_0$-frame and proves the main theorem. Finally, Section \ref{supersimple-sec} studies a sufficient condition to get $(<\aleph_0, \aleph_0)$-locality.

This paper was started while the second author was working on a Ph.D.\ thesis under the direction of Rami Grossberg at Carnegie Mellon University and he would like to thank Professor Grossberg for his guidance and assistance in his research in general and in this work specifically. We also thank Will Boney, Marcos Mazari Armida, and the referee, for their valuable comments on earlier versions of this paper.

Note that at the beginning of several sections, we make global hypotheses assumed throughout the section.

\section{Preliminaries}\label{prelim-sec}

We assume familiarity with the basics of AECs, as presented for example in \cite{grossberg2002, baldwinbook09}, or the first three sections of Chapter I together with the first section of Chapter II in \cite{shelahaecbook}. We also assume familiarity with good frames (see \cite[Chapter II]{shelahaecbook} or \cite{tame-frames-revisited-jsl}; it would help the reader to have a copy of both available during the reading of Section \ref{good-frame-sec}). This section mostly fixes the notation that we will use.

Given a $\tau$-structure $M$, we write $|M|$ for its universe and $\|M\|$ for its cardinality. We may abuse notation and write e.g\ $a \in M$ instead of $a \in |M|$. We may even write $\ba \in M$ instead of $\ba \in \fct{<\omega}{|M|}$.

We write $\K = (K, \lea)$ for an AEC. We may abuse notation and write $M \in \K$ instead of $M \in K$. For a cardinal $\lambda$, we write $\K_\lambda$ for the AEC restricted to its models of size $\lambda$. As shown in \cite[II.1]{shelahaecbook}, any AEC is uniquely determined by its restriction $\K_{\le \LS (\K)}$.

When we say that $M \in \K$ is an \emph{amalgamation base}, we mean (as in \cite{shvi635}) that it is an amalgamation base in $\K_{\|M\|}$, i.e.\ we do \emph{not} require that larger models can be amalgamated.

For $M_0 \in \K$, we say that $M$ is \emph{universal over $M_0$} if $M_0 \lea M$ and for any $N \in \K$ with $M_0 \lea N$, if $\|N\| \le \|M_0\| + \LS (\K)$, there exists $f: N \xrightarrow[M_0]{} M$ (usually we will require also that $\|M\| = \|M_0\|$). We say that $M$ is \emph{$(\lambda, \delta)$-brimmed over $M_0$} (often also called $(\lambda, \delta)$-limit e.g.\ in \cite{shvi635, gvv-mlq}) if $\delta < \lambda^+$ is a limit ordinal, $M_0 \in \K_\lambda$, and there exists an increasing continuous chain $\seq{N_i : i \le \delta}$ of members of $\K_{\lambda}$ such that $N_0$ is universal over $M_0$, $N_\delta = M$, and $N_{i + 1}$ is universal over $N_i$ for all $i < \delta$. We say that $M$ is \emph{brimmed over $M_0$} if it is $(\|M\|, \delta)$-brimmed over $M_0$ for some limit $\delta < \|M\|^+$. We say that $M$ is \emph{brimmed} if it is brimmed over some $M_0$.

The following key concept appears in \cite[I.3.3]{shelahaecbook}:

\begin{defin}\label{sl-def}
  We say that $M \in \K$ is \emph{superlimit} if, letting $\lambda := \|M\|$, we have that $\lambda \ge \LS (\K)$, $M$ is universal in $\K_{\lambda}$ (i.e.\ any $M' \in \K_{\lambda}$ embeds into $M$), $M$ is not maximal, and whenever $\delta < \lambda^+$ is limit, $\seq{M_i : i < \delta}$ is increasing with $M_i \cong M$ for all $i < \delta$, then $\bigcup_{i < \delta} M_i \cong M$.
\end{defin}

The following notion of types already appears in \cite{sh300-orig}. It is called Galois types by many, but we prefer the term \emph{orbital types} here. They are the same types that are defined in \cite[II.1.9]{shelahaecbook}, but we also define them over sets. As pointed out in \cite[Section 2]{sv-infinitary-stability-afml}, this causes no additional difficulties. The following technical point is important: when the AEC does \emph{not} have amalgamation, we may want to compute orbital types only in the subclass of \emph{amalgamation bases} in $\K$ (as in \cite{shvi635}). Thus we allow orbital types to be computed in a subclass of $\K$ in the definition.

\begin{defin}\label{7r.13A}
  Fix an AEC $\K$ and a subclass $\K^\ast$ of $\K$, closed under isomorphisms.

  \begin{enumerate}
  \item We say $(A, N_1, \bb_1) \Eat^{\K^\ast} (A, N_2, \bb_2)$ if:
    \begin{enumerate}
    \item For $\ell = 1,2$, $N_\ell \in \K^\ast$, $A \subseteq |N_\ell|$, and $\bb_\ell \in \fct{<\infty}{|N_\ell|}$.
    \item There exists $N \in \K^\ast$ and $f_\ell : N_\ell \xrightarrow[A]{} N$, $\ell = 1,2$, such that $f_1 (\bb_1) = \bb_2$.
    \end{enumerate}
  \item $\Eat^{\K^\ast}$ is a reflexive and symmetric relation.  Let $E^{\K^\ast}$ be its transitive closure.
  \item Let $\otp_{\K^\ast} (\bb / A; N)$ be the $E^{\K^{\ast}}$-equivalence class of $(\bb, A, N)$.
  \item\label{13A-4} Define $\oSp{\K^\ast} (A, N)$, $\oSp{\K^\ast} (M)$, $\oSp{\K^\ast}^{<\omega} (M)$, $\oSp{\K^\ast}^{<\omega} (\emptyset)$, etc.\ for the Stone spaces of orbital types, \emph{computed inside $\K^\ast$}. This is defined as in \cite[2.20]{sv-infinitary-stability-afml}. For example, $\oSp{\K^\ast}^{<\omega} (\emptyset) = \{\otp_{\K^\ast} (\bb / \emptyset; N) \mid N \in \K^\ast, \bb \in \fct{<\omega}{N}\}$. 
  \item When $\K^\ast = \K$, we omit it.
  \end{enumerate}
\end{defin}

Let us say that an AEC $\K$ is \emph{stable in $\lambda$} if for any $M \in \K_\lambda$, $|\oS (M)| \le \lambda$. This makes sense in any AEC, and is quite well-behaved assuming amalgamation and no maximal models (since then it is known that one can build universal extensions). We will often work in the following axiomatic setup, a slight weakening where full amalgamation is not assumed. This comes from the context derived in \cite{shvi635}:

\begin{defin}\label{nicely-stable-def}
  Let $\K$ be an AEC and let $\lambda$ be a cardinal. We say that $\K$ is \emph{nicely stable in $\lambda$} (or \emph{nicely $\lambda$-stable}) if:

  \begin{enumerate}
  \item $\LS (\K) \le \lambda$.
  \item $\K_\lambda \neq \emptyset$.
  \item $\K$ has joint embedding in $\lambda$.
  \item Density of amalgamation bases: For any $M \in \K_\lambda$, there exists $N \in \K_\lambda$ such that $M \lea N$ and $N$ is an amalgamation base (in $\K_\lambda$).
  \item Existence of universal extensions: For any amalgamation base $M \in \K_\lambda$, there exists an amalgamation base $N \in \K_\lambda$ such that $M \lta N$ and $N$ is universal over $M$.
  \item Any brimmed model in $\K_\lambda$ is an amalgamation base.
  \end{enumerate}

  We say that $\K$ is \emph{very nicely stable in $\lambda$} if in addition it has amalgamation in $\lambda$. 
\end{defin}
\begin{remark}\label{vnice-rmk}
  An AEC $\K$ is very nicely stable in $\lambda$ if and only if $\LS (\K) \le \lambda$, $\K_{\lambda} \neq \emptyset$, $\K$ is stable in $\lambda$, and $\K_\lambda$ has amalgamation, joint embedding, and no maximal models. In particular, stability is a consequence of the existence of universal extensions in Definition \ref{nicely-stable-def}.
\end{remark}

We will repeatedly use the following fact \cite[1.3.6]{shvi635}. 

\begin{fact}\label{brimmed-uq}
  Let $\K$ be nicely stable in $\lambda$ and let $M_0, M_1, M_2 \in \K_{\lambda}$. Let $\delta_1, \delta_2 < \lambda^+$ be limit ordinals such that $\cf{\delta_1} = \cf{\delta_2}$.

  \begin{enumerate}
  \item If $M_\ell$ is $(\lambda, \delta_\ell)$-brimmed over $M_0$, for $\ell = 1,2$, then $M_1 \cong_{M_0} M_2$.
  \item If $M_\ell$ is $(\lambda, \delta_\ell)$-brimmed, for $\ell = 1,2$, then $M_1 \cong M_2$.
  \end{enumerate}
\end{fact}
\begin{proof}
  The first is a straightforward back and forth argument and the second follows from the first using joint embedding.
\end{proof}
\begin{remark}\label{brimmed-uq-rmk}
  Uniqueness of brimmed models when $\cf{\delta_1} \neq \cf{\delta_2}$ is a much harder property to establish, akin to superstability. See for example \cite{shvi635, vandierennomax, gvv-mlq, vandieren-symmetry-apal}. However when $\lambda = \aleph_0$ we automatically have that $\cf{\delta_1} = \cf{\delta_2} = \omega$. 
\end{remark}

Good frames were first defined by the first author in his paper number 600, which eventually appeared as Chapter II of \cite{shelahaecbook}. The idea is to provide a localized (i.e.\ only for base models of a given size $\lambda$) axiomatization of a forking-like notion for a ``nice enough'' set of 1-types.  These axioms are similar to the properties of first-order forking in a superstable theory. Jarden and the first author (in \cite{jrsh875}) later gave a slightly more general definition, not assuming the existence of a superlimit model and dropping some of the redundant clauses. We will make use of good frames for types of finite length (not just length one). Their definition is just like for types of length one, we call them \emph{good $(<\omega, \lambda)$-frames}. For the convenience of the reader, we give the full definition from \cite[3.8]{tame-frames-revisited-jsl} here:

\begin{defin}\label{good-frame-def}
  Let $\lambda$ be an infinite cardinal. A \emph{good $(<\omega, \lambda)$-frame} is a triple $(\K, \nf, \Sbs)$ satisfying, where:

  \begin{enumerate}
  \item $\K$ is an abstract elementary class with $\lambda \ge \LS (\K)$, $\K_\lambda \neq \emptyset$.
  \item $\Sbs \subseteq \bigcup_{M \in \K_{\lambda}} \Ss^{<\omega} (M)$. Moreover, if $\otp (\seq{a_i : i < n} / M; N) \in \Sbs (M)$, then $a_i \notin M$ for all $i < n$.
  \item $\nf$ is a relation on quadruples of the form $(M_0, M_1, \ba, N)$, where $M_0 \lea M_1 \lea N$, $\ba \in \fct{<\omega}{N}$, and $M_0$, $M_1$, $N$ are all in $K_{\lambda}$. We write $\nf(M_0, M_1, \ba, N)$ or $\nfs{M_0}{\ba}{M_1}{N}$ instead of $(M_0, M_1, a, N) \in \nf$. 
  \item The following properties hold:
    \begin{enumerate}
      \item \underline{Invariance}: If $f: N \cong N'$ and $\nfs{M_0}{\ba}{M_1}{N}$, then $\nfs{f[M_0]}{f(\ba)}{f[M_1]}{N'}$. If $\tp (\ba / M_1; N) \in \Sbs (M_1)$, then $\tp(f (\ba) / f[M_1]; N') \in \Sbs (f[M_1])$.
      \item \underline{Monotonicity}: If $\nfs{M_0}{\ba}{M_1}{N}$, $\ba'$ is a subsequence of $\ba$, $M_0 \lea M_0' \lea M_1' \lea M_1 \lea N' \lea N \lea N''$ with $\ba' \in N'$, and $N'' \in K_{\F}$, then $\nfs{M_0'}{\ba'}{M_1'}{N'}$ and $\nfs{M_0'}{\ba'}{M_1'}{N''}$. If $\tp (\ba / M_1; N) \in \Sbs (M_1)$ and $\ba'$ is a subsequence of $\ba$, then $\tp (\ba' / M_1; N) \in \Sbs (M_1)$. [This property and the previous one show that $\nf$ is really a relation on types. Thus if $p \in \Ss^{<\omega} (M_1)$ is a type, we say $p$ \emph{does not fork over $M_0$} if $\nfs{M_0}{\ba}{M_1}{N}$ for some (equivalently any) $\ba$ and $N$ such that $p = \tp (\ba / M_1; N)$. Note that this depends on $\s$, but $\s$ will always be clear from context.]
      \item \underline{Nonforking types are basic}: If $\nfs{M}{\ba}{M}{N}$, then $\tp (\ba / M; N) \in \Sbs (M)$.
        
      \item $K_{\lambda}$ has amalgamation, joint embedding, and no maximal models.
      \item \underline{bs-Stability}: $|\Sbs (M)| \le \|M\|$ for all $M \in K_{\lambda}$.
      \item \underline{Density of basic types}: If $M \lta N$ are in $K_{\lambda}$, then there is $a \in N$ such that $\tp (a / M; N) \in \Sbs (M)$.
  \item \underline{Existence of nonforking extension}: If $m \le n < \omega$, $p \in \Sbs (M) \cap \Ss^m (M)$, $N \gea M$ is in $K_{\lambda}$, then there is some $q \in \Sbs(N) \cap \Ss^n (M)$ that does not fork over $M$ and extends $p$.
  \item \underline{Uniqueness}: If $p, q \in \Ss^{<\omega} (N)$ do not fork over $M$ and $p \upharpoonright M = q \upharpoonright M$, then $p = q$.
  \item \underline{Symmetry}: If $\nfs{M_0}{\ba_1}{M_2}{N}$, $\ba_2 \in \fct{<\alpha}{M_2}$, and $\tp (\ba_2 / M_0; N) \in \Sbs (M_0)$, then there is $M_1$ containing $\ba_1$ and $N' \gea N$ such that $\nfs{M_0}{\ba_2}{M_1}{N'}$.
  \item \underline{Local character}: If $\delta$ is a regular cardinal, $\seq{M_i \in K_{\lambda} : i \le \delta}$ is increasing continuous, and $p \in \Sbs (M_\delta)$ is such that $\ell (p) < \delta$, then there exists $i < \delta$ such that $p$ does not fork over $M_i$.
  \item \underline{Continuity}: If $\delta$ is a limit ordinal, $\seq{M_i \in K_{\lambda}: i \leq \delta}$ and $\seq{\alpha_i < \alpha : i \leq \delta}$ are increasing and continuous, and $p_i \in \Sbs (M_i)$ for $i < \delta$ are such that $j < i < \delta$ implies $p_j = p_i \rest M_j$, then there is some $p \in \Sbs (M_\delta)$ such that for all $i < \delta$, $p_i = p \rest M_i$.  Moreover, if each $p_i$ does not fork over $M_0$, then neither does $p$.
  \item \underline{Transitivity}: If $M_0 \lea M_1 \lea M_2$, $p \in \Ss (M_2)$ does not fork over $M_1$ and $p \upharpoonright M_1$ does not fork over $M_0$, then $p$ does not fork over $M_0$.
    \end{enumerate}
\end{enumerate}
    A \emph{good $\lambda$-frame} is defined similarly, except we require all types to be types of singletons (i.e.\ they are in $\Ss (M)$ instead of $\Ss^{<\omega} (M)$). We say that an AEC $\K$ \emph{has a good $(<\omega, \lambda)$-frame} if there is a good $(<\omega, \F)$-frame where $K$ is the underlying AEC.
  \end{defin}
  \begin{notation}
    If $\s = (\K, \nf, \Sbs)$ is a good-$(<\omega, \lambda)$-frame (or a good $\lambda$-frame), write $\nf_{\s} := \nf$. Also write $\K^{\s}$ for $\K$ and $\K_{\s} = \K_\lambda$. We will also write $M \leas N$ as a shortcut for $M \lea N$ and $M, N \in \K_{\s}$ ($= \K_\lambda$).
  \end{notation}

\begin{remark}
  The reader might wonder about the reasons for having a special class of basic types. Following \cite[Definition III.9.2]{shelahaecbook}, let us call a good frame \emph{type-full} if the basic types are all the nonalgebraic types. There are no known examples of a good $\lambda$-frame which which cannot be extended to a type-full one. However a construction of good frames of the first author \cite[II.3]{shelahaecbook} builds a non type-full good frame and it is not clear that it can be extended to a type-full one until a lot more machinery has been developed. Thus it can be easier to build a good frame than to build a type-full one, and most results about frames already hold in the non-type-full context. That being said, readers would not miss the essence of the present paper if they assumed that all the frames here were type-full.
\end{remark}
\begin{remark}\label{indep-seq-rmk}
  Any good $\lambda$-frame (i.e.\ for types of length one) extends to a good $(<\omega, \lambda)$-frame (using independent sequences, see \cite[III.9.4]{shelahaecbook}) or \cite[5.8]{tame-frames-revisited-jsl}. This frame will however \emph{not} be type-full.
\end{remark}

From now on until the end of Section \ref{building-aleph0-frame}, ``nonforking'' will refer to nonforking in a fixed frame $\s$ (usually clear from context).

\section{Weak nonforking amalgamation}\label{good-frame-sec}

In this section, we work in a good $\lambda$-frame and study a natural weak version of nonforking amalgamation, $\LWNF_{\s}$ ($\LWNF$ stands for ``left weak nonforking amalgamation''). The goal is to obtain a natural criteria for proving the existence of a superlimit in $\aleph_1$ and also prepare the ground for the proof of symmetry in the good frame built in Section \ref{building-aleph0-frame}. The main results are the existence property (Theorem \ref{weak-props}) and how the symmetry property of $\LWNF_{\s}$ is connected to $\s$ being $\goodp$ (Theorem \ref{sym-equiv}). Throughout this section, we assume:

\begin{hypothesis}\label{global-hyp-3} \
  \begin{enumerate}
  \item $\s = (\K, \nf, \Sbs)$ is a fixed good $(<\omega, \lambda)$-frame, except that it may not satisfy the symmetry axiom.
  \item $\K$ is categorical in $\lambda$.
  \end{enumerate}
\end{hypothesis}
\begin{remark}
  In this section, $\lambda$ is allowed to be uncountable. However the case $\lambda = \aleph_0$ is the one that will interest us in the next sections.
\end{remark}

The reason for not assuming symmetry is that we will use some of the results of this section to \emph{prove} that the symmetry axiom holds of a certain nonforking relation in Section \ref{building-aleph0-frame}.

We will use:

\begin{fact}[II.4.3 in \cite{shelahaecbook}]\label{brimmed-fact}
  Let $\delta < \lambda^+$ be a limit ordinal divisible by $\lambda$. Let $\seq{M_i : i \le \delta}$ be increasing continuous in $\K_{\s}$. If for any $i < \delta$ and any $p \in \oSsbs (M_i)$, there exists $\lambda$-many $j \in [i, \delta)$ such that the nonforking extension of $p$ to $M_j$ is realized in $M_{j + 1}$, then $M_\delta$ is brimmed over $M_0$.
\end{fact}

To understand the definition below, it may be helpful to think of $\s$ as type-full. Then $\LWNF_{\s} (M_0, M_1, M_2, M_3)$ holds if and only if the type of any finite subsequences of $M_1$ over $M_2$ does not fork over $M_0$ ($M_3$ is the ambient model). Thus $\LWNF_{\s}$ is an attempt to extend nonforking to types of infinite sequences so that it keeps a strong finite character property. In the present paper, $\LWNF_{\s}$ will be a helpful technical tool but it is not clear that it has the uniqueness property (in contrast with the relation $\NF$ from \cite[\S II.6]{shelahaecbook} or \cite[\S5]{jrsh875}, which will have the uniqueness property but requires more assumptions on the good frame). If $\LWNF_{\s}$ \emph{does} have the uniqueness property, this has strong consequence on the structure of the frame, see Theorem \ref{uq-consequence}.

\begin{defin}\label{wnf-def}
  Define the following $4$-ary relations on $\K_{\s}$:

  \begin{enumerate}
  \item $\LWNF_{\s} (M_0, M_1, M_2, M_3)$ if and only if $M_0 \leas M_\ell \leas M_3$ for $\ell = 1,2$ and for any $\bb \in \fct{<\omega}{|M_1|}$, if $\tp (\bb, M_2, M_3)$ and $\tp (\bb, M_0, M_3)$ are basic then $\tp (\bb, M_2, M_3)$ does not fork over $M_0$.
  \item $\RWNF_{\s} (M_0, M_1, M_2, M_3)$ if and only if $\LWNF_{\s} (M_0, M_2, M_1, M_3)$ [$\RWNF$ stands for ``right weak nonforking amalgamation''].
  \item $\WNF_{\s} (M_0, M_1, M_2, M_3)$ if and only if both $\LWNF_{\s} (M_0, M_1, M_2, M_3)$ and $\RWNF_{\s} (M_0, M_1, M_2, M_3)$ [$\WNF$ stands for ``weak nonforking amalgamation].
  \end{enumerate}

  When $\s$ is clear from context, we write $\LWNF$, $\RWNF$, and $\WNF$.
\end{defin}

The following result often comes in handy.

\begin{lem}\label{union-realize}
  Let $\delta < \lambda^{+}$ be a limit ordinal. Let $\seq{M_i : i \le \delta}$, $\seq{N_i : i \le \delta}$ be increasing continuous in $\K_{\s}$. Assume that for each $i \le j < \delta$, we have that $\LWNF (M_i, N_i, M_j, N_j)$. If for each $i < \delta$, $N_i$ realizes all the basic types over $M_i$, then $N_\delta$ realizes all the basic types over $M_\delta$.
\end{lem}
\begin{proof}
  Let $p \in \oSsbs (M_\delta)$. By local character, there exists $i < \delta$ such that $p$ does not fork over $M_i$. In particular, $p \rest M_i$ is basic. Since $N_i$ realizes all the basic types over $M_i$, there exists $a \in |N_i|$ such that $p \rest M_i = \tp (a, M_i, N_i)$. Because for all $j \in [i, \delta)$, $\LWNF (M_i, N_i, M_j, N_j)$, we have by continuity that $\tp (a, M_\delta, N_\delta)$ does not fork over $M_i$, hence by uniqueness it must be equal to $p$. Therefore $a$ realizes $p$, as needed.
\end{proof}

Next, we investigate the properties of $\LWNF$. We are especially interested in the symmetry property: whether $\LWNF$ is equal to $\RWNF$. To understand it better, we consider the following ordering, defined similarly to $\le_{\lambda^+}^\ast$ from \cite[II.7.2]{shelahaecbook}:

\begin{defin}\label{leanf-def}
  For $R \in \{\LWNF, \RWNF, \WNF\}$, define a relation $\leap{R}$ on $\K_{\lambda^+}$ as follows. For $M^0, M^1 \in \K_{\lambda^+}$, $M^0 \leap{R} M^1$ if and only if there exists increasing continuous resolutions $\seq{M_i^\ell \in \K_\lambda : i < \lambda^+}$ of $M^\ell$ for $\ell = 0,1$ such that for all $i < j < \lambda^+$, $R (M_i^0, M_i^1, M_j^0, M_j^1)$.
\end{defin}

The following is a straightforward ``catching your tail argument'', see the proof of \cite[4.6]{downward-categ-tame-apal} (this assumes that all types are basic, but the argument goes through without this restriction). Roughly, it says that if $M \lea N$ ($\lea$ is the usual order on $\K$), then we can find a resolution of $M$ and $N$ so that the pieces are in \emph{left} weak nonforking amalgamation.

\begin{fact}\label{reflects-lwnf}
  Let $M, N \in \K_{\lambda^+}$. If $M \lea N$, then $M \leap{\LWNF} N$.
\end{fact}

Whether $M \leap{\RWNF} N$ can be concluded as well seems to be a  much more complicated question, and in fact is equivalent to $\s$ being $\goodp$ (Theorem \ref{sym-equiv}), a weakening of symmetry. We now observe that an increasing union of a $\leap{\RWNF}$-increasing chain of saturated models is saturated:

\begin{lem}\label{chainsat-lem}
  Let $\delta < \lambda^{++}$ be a limit ordinal. If $\seq{M_i : i < \delta}$ is a $\leap{\RWNF}$-increasing sequence of saturated models in $\K_{\lambda^+}$, then $\bigcup_{i < \delta} M_i$ is saturated.
\end{lem}
\begin{proof}
  If $\cf{\delta} \ge \lambda^+$, then any $\bigcup_{i < \delta} M_i$ will be $\lambda^+$-saturated on general grounds. Thus assume without loss of generality that $\delta = \cf{\delta} < \lambda^+$. Let $M_\delta := \bigcup_{i < \delta} M_i$. We build $\seq{M_{i, j} : i \le \delta, j \le \lambda^+}$ such that:

  \begin{enumerate}
  \item For any $i \le \delta$, $M_{i, \lambda^+} = M_i$.
  \item For any $i < \delta, j < \lambda^+$, $M_{i, j} \in \K_{\s}$.
  \item For any $i \le \delta$, $\seq{M_{i, j} : j < \lambda^+}$ is increasing and continuous.
  \item For any $j \le \lambda^+$, $\seq{M_{i, j} : i < \delta}$ is increasing and $M_{\delta, j} = \bigcup_{i < \delta} M_{i, j}$.
  \item For any $i_1 < i_2 \le \delta$, $j_1 < j_2 \le \lambda^+$, $M_{i_2, j_2}$ realizes all the types in $\oSsbs (M_{i_1, j_1})$.
  \end{enumerate}

  This is easy to do. Now for each $i_1 < i_2 < \delta$, we have by assumption that $M_{i_1} \leap{\RWNF} M_{i_2}$. Thus the set $C_{i_1, i_2}$ of $j < \lambda^+$ such that for all $j' \in [j, \lambda^+)$, $\RWNF (M_{i_1, j}, M_{i_2, j}, M_{i_1, j'}, M_{i_2, j'})$ is a club (that it is closed follows from the local character and continuity axioms of good frames). Therefore $C := \bigcap_{i_1 < i_2 < \delta} C_{i_1, i_2}$ is also a club. Hence by renaming without loss of generality for all $i_1 < i_2 < \delta$ and all $j \le j' < \lambda^+$, $\RWNF (M_{i_1, j}, M_{i_2, j}, M_{i_1, j'}, M_{i_2, j'})$.

    Now let $N \lea M_\delta$ be such that $N \in \K_{\lambda}$. We want to see that any type over $N$ is realized in $M_\delta$. By Fact \ref{brimmed-fact}, it is enough to show that any \emph{basic} type over $N$ is realized in $M_\delta$.

    Let $j < \lambda^+$ be big-enough such that $N \lea M_{\delta, j}$. It is enough to see that any basic type over $M_{\delta, j}$ is realized in $M_{\delta, j + 1}$. To see this, use Lemma \ref{union-realize} with $\seq{M_i : i \le \delta}$, $\seq{N_i : i \le \delta}$ there standing for $\seq{M_{i, j} : i \le \delta}$, $\seq{M_{i, j + 1} : i \le \delta}$ here. We know that for each $i \le i' < \delta$, $\RWNF (M_{i, j}, M_{i', j}, M_{i, j + 1}, M_{i', j + 1})$ and therefore $\LWNF (M_{i, j}, M_{i, j  +1}, M_{i', j}, M_{i', j + 1})$. Thus the hypotheses of Lemma \ref{union-realize} are satisfied.
\end{proof}

The next fact will be used to prove the existence property of $\LWNF$. Its proof is a direct limit argument similar to e.g.\ \cite[5.2]{gvv-mlq}. Roughly, the nonforking relation there is given by ``there exists a smaller submodel over which the type does not \emph{split}''; in fact, these smaller submodels have to be kept as part of the data of the tower. This is not needed here. The argument is also similar to \cite[3.1.8]{jrsh875}. However there the symmetry axiom axiom is needed: there is an extra requirement on the type of a certain element $b$, but here we do \emph{not} make that requirement so do \emph{not} need symmetry.

\begin{fact}\label{tower-ext}
  Let $\alpha < \lambda^+$. Let $\seq{M_i : i \le \alpha}$ be $\leas$-increasing continuous (in $\K_{\lambda}$) and let $\seq{\ba_i : i < \alpha}$ be given such that $\ba_i \in \fct{<\omega}{M_{i + 1}}$ for all $i < \alpha$ and $\tp (\ba_i, M_i, M_{i + 1}) \in \oSsbs (M_i)$ (we allow the $\ba_i$'s to have different length).

  There exists $\seq{N_i : i \le \alpha}$ $\leas$-increasing continuous such that:

  \begin{enumerate}
  \item $M_i \ltas N_i$ for all $i \le \alpha$.
  \item $\tp (\ba_i, N_i, N_{i + 1})$ does not fork over $M_i$.
  \end{enumerate}
\end{fact}

We can now list and then prove some basic properties of weak nonforking amalgamation. For the convenience of the reader, we repeat Hypothesis \ref{global-hyp-3}.

\begin{thm}\label{weak-props}
  Let $\s = (\K, \nf, \Sbs)$ be a fixed good $(<\omega, \lambda)$-frame, except that it may not satisfy the symmetry axiom. Assume that $\K$ is categorical in $\lambda$. Let $R \in \{\LWNF, \RWNF, \WNF\}$.
  
  \begin{enumerate}
  \item Invariance: If $R (M_0, M_1, M_2, M_3)$ and $f: M_3 \cong M_3'$, then $R (f[M_0], f[M_1], f[M_2], M_3')$.
  \item Monotonicity: If $R (M_0, M_1, M_2, M_3)$ and $M_0 \leas M_\ell' \leas M_\ell$ for $\ell = 1,2$, then $R (M_0, M_1', M_2', M_3)$.
  \item Ambient monotonicity: If $R (M_0, M_1, M_2, M_3)$ and $M_3 \leas M_3'$, then $R (M_0, M_1, M_2, M_3')$. If $M_3'' \leas M_3$ contains $|M_1| \cup |M_2|$, then $R (M_0, M_1, M_2, M_3'')$.
  \item Continuity: If $\delta < \lambda^+$ is a limit ordinal and $\seq{M_i^\ell : i \le \delta}$ are increasing continuous for $\ell < 4$ with $R (M_i^0, M_i^1, M_i^2, M_i^3)$ for each $i < \delta$, then $R (M_\delta^0, M_\delta^1, M_\delta^2, M_\delta^3)$.
  \item Long transitivity: If $\alpha < \lambda^+$ is an ordinal, $\seq{M_i : i \le \alpha}$, $\seq{N_i : i \le \alpha}$ are increasing continuous and $\LWNF (M_i, N_i, M_{i + 1}, N_{i + 1})$ for all $i < \alpha$, then $\LWNF (M_0, N_0, M_\alpha, N_\alpha)$.
  \item Existence: If $R \neq \WNF$, $M_0 \leas M_\ell$, $\ell = 1,2$, then there exists $M_3 \in \K_\lambda$ and $f_\ell: M_\ell \xrightarrow[M_0]{} M_3$ such that $R (M_0, f_1[M_1], f_2[M_2], M_3)$.
  \end{enumerate}
\end{thm}
\begin{proof}
  Invariance and the monotonicity properties are straightforward to prove. Continuity and long transitivity follow directly from the local character, continuity, and transitivity properties of good frames. We prove existence via the following claim:

  \underline{Claim}: There exists $N_0, N_1, N_2, N_3 \in \K_{\s}$ such that $\LWNF (N_0, N_1, N_2, N_3)$ and $N_\ell$ is brimmed over $N_0$ for $\ell = 1,2$.

  Existence easily follows from the claim: given $M_0 \leas M_\ell$, $\ell = 1,2$, there is (by categoricity in $\lambda$) an isomorphism $f: M_0 \cong N_0$ and (by universality of brimmed models) embeddings $f_\ell : M_\ell \rightarrow N_\ell$ extending $f$ for $\ell = 1,2$. After some renaming, we obtain the desired $\LWNF$-amalgam. To obtain an $\RWNF$-amalgam, reverse the role of $M_1$ and $M_2$.

  \underline{Proof of Claim}: The idea of the proof is as follows: for some suitable ordinal $\alpha$, we want to build $\seq{M_i : i \le \alpha}$, $\seq{a_i \in M_{i + 1} : i < \alpha}$ with the following property: whenever $\seq{N_i : i \le \alpha}$ is as described by Fact \ref{tower-ext} (plus slightly more), we must have that $\LWNF (M_0, M_\alpha, N_0, N_\delta)$, $M_\alpha$ is brimmed over $M_0$, and $N_0$ is brimmed over $M_0$. To achieve this, we simply start with an arbitrary $\seq{M_i : i \le \alpha}$, $\seq{a_i : i < \alpha}$ and, if it fails the property, take a witness to the failure, add some more $a_j$'s to make it more brimmed, and start again to consider whether this witness satisfies the property. After doing this for sufficiently many steps, we eventually succeed to build the desired object. This is somewhat similar to the construction of a reduced tower in \cite{shvi635, gvv-mlq}, although here we are dealing with nonforking independence and not just set-theoretic disjointness.

  We now start with the proof. Let $\delta := \lambda \cdot \lambda$. We choose $(\bar M^\alpha,\bar a^\alpha)$ by induction on $\alpha \le \delta$ such that:
  
  \begin{enumerate}
  \item $\bM^\alpha = \seq{M^\alpha_i :i \le \alpha}$ is $\leas$-increasing continuous.

  \item $\ba^\alpha = \seq{\ba_i : i < \alpha}$, and $\ba_i \in M_{i + 1}^\alpha$ for all $i < \alpha$.
    \item For all $i < \alpha$, $\tp (\ba^{\alpha}_i, M^\alpha_i, M^\alpha_{i+1}) \in \oSsbs (M_i^\alpha)$.
    \item For each $i \le \delta$, $\seq{M^\alpha_i:\alpha \in [i,\delta]}$ is $\ltas$-increasing continuous.
    \item\label{clause-5} For each $i < \delta$ and each $\alpha \in (i, \delta]$, $\tp (\ba_i, M_i^\alpha, M_{i + 1}^{\alpha})$ does not fork over $M_i^{i}$.
\item\label{clause-6} If $p \in \oSsbs (M_i^\alpha)$ for $i \le \alpha < \delta$, then for $\lambda$-many $\beta \in [\alpha,\delta)$, $\tp( \ba_\beta,M^{\beta + 1}_\beta,M^{\beta +1}_{\beta + 1})$ is a nonforking extension of $p$.
\item\label{clause-7} If $i < \alpha < \delta$ and $\tp (\ba, M_0^\alpha, M_{i + 1}^{\alpha}) \in \oSsbs (M_0^\alpha)$, then for some $\beta \in (\alpha,\delta)$ exactly one of the following occurs:
  \begin{enumerate}
  \item\label{clause-7a} $\tp(\bar a,M^{\beta +1}_0,M^{\beta+1}_{i+1})$ forks over $M^\alpha_0$.
  \item\label{clause-7b} There is no $\seq{M^*_j:j \le i+1}$ $\leas$-increasing continuous such that:
    \begin{enumerate}
    \item $M_j^\beta \leas M_j^*$ for all $j \le i + 1$.
    \item $\tp (\ba_j, M_j^*, M_{j + 1}^*)$ does not fork over $M_j^\beta$ for all $j < i + 1$.
    \item $\tp (\ba, M_0^*, M_{i + 1}^*)$ forks over $M^\beta_0$.
    \end{enumerate}
  \end{enumerate}
\end{enumerate}

  \begin{itemize}
  \item[] \underline{This is possible}:
    Along the construction, we also build an enumeration $\seq{(\bb_j^\gamma, k_j^\gamma, i_j^\gamma, \alpha_j^\gamma) : j < \lambda, \gamma < \lambda}$ such that for any $\gamma \in (0, \lambda)$, any $\alpha < \lambda \cdot \gamma$, any $i < \alpha$, any $k \le i$, and any $\ba \in \fct{<\omega}{M_{i + 1}^\alpha}$, if $\tp (\ba, M_k^\alpha, M_{i + 1}^\alpha) \in \oSsbs (M_k^\alpha)$, then there exists $j < \lambda$ so that $\bb_j^{\gamma} = \ba$, $i_j^{\gamma} = i$, $k_j^\gamma = k$, and $\alpha_j^{\gamma} = \alpha$. We require that always $k_j^\gamma \le i_j^{\gamma} < \alpha_j^{\gamma} < \lambda \cdot \gamma$ and the triple $(\bb_j^{\gamma}, M_{k_j^{\gamma}}^{\alpha_j^{\gamma}}, M_{i_j^\gamma + 1}^{\alpha_j^{\gamma}})$ represents a basic type. We make sure that at stage $\lambda \cdot (\gamma + 1)$ of the construction below, $\bb_j^{\gamma'}, k_j^{\gamma'}, i_j^{\gamma'}, \alpha_j^{\gamma'}$ are defined for all $j < \lambda$, $\gamma' \le \gamma$.
    
    For $\alpha = 0$, take any $M_0^0 \in \K_{\s}$. For $\alpha$ limit, let $M_i^\alpha := \bigcup_{\beta \in [i, \alpha)} M_i^\beta$ for $i < \alpha$ and $M_\alpha^\alpha := \bigcup_{i < \alpha} M_i^\alpha$. Now assume that $\bM^{\alpha}$, $\ba^{\alpha}$ have been defined  for $\alpha < \delta$. We define $\bM^{\alpha + 1}$ and $\ba_{\alpha}$. Fix $\rho$ and $j < \lambda$ such that $\alpha = \lambda \cdot \rho + j$. We consider two cases.

      \begin{itemize}
      \item \underline{Case 1: $\rho$ is zero or a limit}: Use Fact \ref{tower-ext} to get $\seq{M_i^{\alpha + 1} : i \le \alpha}$ $\ltas$-increasing continuous such that $M_i^\alpha \ltas M_i^{\alpha + 1}$ for all $i \le \alpha$, and for all $i < \alpha$, $\tp (\ba_i, M_i^{\alpha + 1}, M_{i + 1}^{\alpha + 1})$ does not fork over $M_i^\alpha$. Pick any $M_{\alpha + 1}^{\alpha + 1}$ with $M_\alpha^{\alpha + 1} \ltas M_{\alpha + 1}^{\alpha + 1}$ and any $\ba_{\alpha} \in \fct{<\omega}{M_{\alpha + 1}^{\alpha + 1}}$ such that $\tp (\ba_{\alpha}, M_{\alpha}^{\alpha + 1}, M_{\alpha + 1}^{\alpha + 1}) \in \oSsbs (M_\alpha^{\alpha + 1})$.
      \item \underline{Case 2: $\rho$ is a successor}: Say $\rho = \gamma + 1$. Let $\ba := \bb_j^\gamma$, $\alpha_0 := \alpha_j^\gamma$, $k_0 := k_j^\gamma$, $i_0 := i_j^\gamma$. There are two subcases.
        \begin{itemize}
        \item \underline{Subcase 1}: Either $k_0 \neq 0$, or $k_0 = 0$ and (\ref{clause-7b}) holds with $i, \alpha, \beta$ there standing for $i_0, \alpha_0, \alpha$ here.

          In this case, we proceed as in Case 1 to define $\seq{M_i^{\alpha + 1} : i \le \alpha}$. Then we pick $\ba_{\alpha}$, $M_{\alpha + 1}^{\alpha + 1}$ such that $\tp (\ba_\alpha, M_\alpha^{\alpha}, M_{\alpha + 1}^{\alpha + 1})$ is the nonforking extension of $\tp (\ba, M_{i_0}^{\alpha_0}, M_{i_0 + 1}^{\alpha_0})$.

        \item \underline{Subcase 2}: $k_0 = 0$ and (\ref{clause-7b}) fails.

          In this case, let $\seq{M_j^\ast : j \le i_0 + 1}$ witness the failure and set $M_j^{\alpha + 1} := M_j^\ast$ for $j \le i_0 + 1$. Then continue as in Case 1 and define $\ba_\alpha$, $M_{\alpha + 1}^{\alpha + 1}$ as before.
        \end{itemize}
      \end{itemize}
    \item[] \underline{This is enough}:

      We proceed via a series of subclaims:

      \underline{Subclaim 1}: If $p \in \oSsbs (M_i^\delta)$ for $i < \delta$, then for $\lambda$-many $\beta \in [i,\delta)$, $\tp( \ba_\beta,M_{\beta}^{\delta},M^{\delta}_{\beta + 1})$ is a nonforking extension of $p$.

        \underline{Proof of subclaim 1}: Pick $i' \in (i, \delta)$ such that $p$ does not fork over $M_i^{i'}$. By (\ref{clause-6}), we know that for $\lambda$-many $\beta \in [i', \delta)$, the nonforking extension of $p \rest M_i^{i'}$ to $M_{\beta}^{\beta + 1}$ is realized in $M_{\beta + 1}^{\beta + 1}$ by $\ba_\beta$. But by (\ref{clause-5}) we also have that $\tp (\ba_\beta, M_{\beta}^{\delta}, M_{\beta + 1}^{\delta})$ does not fork over $M_\beta^{\beta}$. In particular by uniqueness $\ba_\beta$ also realizes $p$. $\dagger_{\text{subclaim 1}}$

          \underline{Subclaim 2}: $M_\delta^{\delta}$ is brimmed over $M_0^\delta$.
          
          \underline{Proof of subclaim 2}: Apply Fact \ref{brimmed-fact} to the chain $\seq{M_i^\delta : i \le \delta}$, using the previous step. $\dagger_{\text{subclaim 2}}$

          We now choose $\bM^* = \seq{M_i^\ast : i \le \delta}$ increasing continuous such that $M_0^\ast$ is brimmed over $M_0^\delta$, $M_i^\delta \leas M_i^\ast$ for all $i \le \delta$, and $\tp (\ba_i, M_i^\ast, M_{i + 1}^\ast)$ does not fork over $M_i^\delta$. This is possible, see case 1 above. Now let $(N_0, N_1, N_2, N_3) := (M_0^\delta, M_\delta^\delta, M_0^\ast, M_\delta^\ast)$. We have just said that $M_0^\ast$ is brimmed over $M_0^\delta$, and by subclaim 2, $M_\delta^{\delta}$ is brimmed over $M_0^\delta$. Thus $N_\ell$ is brimmed over $N_0$ for $\ell = 1,2$. It remains to see:

          \underline{Subclaim 3}: $\LWNF (M_0^\delta, M_\delta^{\delta}, M_0^\ast, M_{\delta}^\ast)$

          \underline{Proof of subclaim 3}: Pick $\ba \in \fct{<\omega}{M_\delta^{\delta}}$ such that $\tp (\ba, M_0^\delta, M_{\delta}^\delta)$ is basic. By local character, there exists $\alpha < \delta$ such that $\tp (\ba, M_0^\delta, M_\delta^\delta)$ does not fork over $M_0^\alpha$. Further, we can increase $\alpha$ if necessary and pick $i < \alpha$ such that $\ba \in \fct{<\omega}{M_{i + 1}^\alpha}$. We now apply Clause (\ref{clause-7}). We know that (\ref{clause-7a}) fails for all $\beta \in (\alpha, \delta)$ by the choice of $\alpha$, therefore (\ref{clause-7b}) must hold for all $\beta \in (\alpha, \delta)$. Now if $\tp (\ba, M_0^\ast, M_{\delta}^\ast)$ forks over $M_0^\delta$, then it must fork over $M_0^\beta$ for all high-enough $\beta$, but then $\seq{M_j^\ast : j \le i + 1}$ would contradict Clause (\ref{clause-7b}). Therefore $\tp (\ba, M_0^\ast, M_{\delta}^\ast)$ does not fork over $M_0^\delta$, as desired. $\dagger_{\text{subclaim 3}}$

          $\dagger_{\text{Claim}}$
  \end{itemize}
\end{proof}

The following properties of $\LWNF$ may or may not hold in general (we have no examples for the failure of symmetry, but uniqueness fails in the last good frame of the Hart-Shelah example, see \cite{hs-example,counterexample-frame-v4-toappear}):

\begin{defin}\label{sym-def}
  Let $R \in \{\LWNF, \RWNF, \WNF\}$.

  \begin{enumerate}
  \item We say that $R$ has the \emph{symmetry property} if $R (M_0, M_1, M_2, M_3)$ implies $R (M_0, M_2, M_1, M_3)$.
  \item We say that $R$ has the \emph{uniqueness property} if whenever $R (M_0, M_1, M_2, M_3)$ and $R (M_0, M_1, M_2, M_3')$, there exists $M_3''$ with $M_3' \leas M_3''$ and $f: M_3 \xrightarrow[|M_1| \cup |M_2|]{} M_3''$.
  \end{enumerate}
\end{defin}

The following are trivial observations about the definitions:

\begin{remark}\label{trivial-obs} \
  \begin{enumerate}
  \item $\WNF$ has the symmetry property, and $\LWNF$ has the symmetry property if and only if $\RWNF$ has the symmetry property if and only if $\LWNF = \RWNF = \WNF$.
  \item $\LWNF$ has the uniqueness property if and only $\RWNF$ has it.
  \end{enumerate}
\end{remark}

Recall from \cite[III.1.3]{shelahaecbook}:

\begin{defin}\label{goodp-def}
  $\s$ is \emph{$\goodp$} when the following is \emph{impossible}:

  There exists an increasing continuous $\seq{M_i : i < \lambda^+}$, $\seq{N_i : i < \lambda^+}$, a basic type $p \in \oSsbs (M_0)$, and $\seq{\ba_i : i < \lambda^+}$ such that for any $i < \lambda^+$:

  \begin{enumerate}
  \item $M_i \leas N_i$.
  \item $\ba_{i + 1} \in |M_{i + 2}|$ and $\tp (\ba_{i + 1}, M_{i + 1}, M_{i + 2})$ is a nonforking extension of $p$, but $\tp (\ba_{i + 1}, N_0, N_{i + 2})$ forks over $M_0$.
  \item $\bigcup_{j < \lambda^+} M_j$ is saturated.
  \end{enumerate}
\end{defin}

We now show that being $\goodp$ is a consequence of symmetry for $\LWNF$. Moreover, $\goodp$ allows us to build a superlimit in $\lambda^+$.

\begin{thm}\label{sym-equiv}
  $(\ref{equiv-0}) \Rightarrow (\ref{equiv-1}) \Leftrightarrow (\ref{equiv-2}) \Rightarrow (\ref{equiv-3})$, where:

  \begin{enumerate}
  \item\label{equiv-0} $\LWNF$ has the symmetry property.
  \item\label{equiv-1} $\s$ is $\goodp$.
  \item\label{equiv-2} For $M, N \in \K_{\lambda^+}$ both saturated, $M \lea N$ implies $M \leap{\WNF} N$.

  \item\label{equiv-3} There is a superlimit model in $\K_{\lambda^+}$.
  \end{enumerate}
\end{thm}
\begin{proof} \
  \begin{itemize}
    \item \underline{(\ref{equiv-2}) implies (\ref{equiv-3})}: This follows from Lemma \ref{chainsat-lem} and the fact that the saturated model in $\lambda^+$ is universal and has a proper extension \cite[II.4.13]{shelahaecbook}.

    \item \underline{$\neg (\ref{equiv-1})$ implies $\neg (\ref{equiv-2})$}: Fix a witness $\seq{M_i : i < \lambda^+}$, $\seq{N_i : i < \lambda^+}$, $\seq{\ba_i : i < \lambda^+}$, $p$ to the failure of being $\goodp$. Write $M_{\lambda^+} := \bigcup_{i < \lambda^+} M_i$, $N_{\lambda^+} := \bigcup_{i < \lambda^+} N_i$. By assumption, $M_{\lambda^+}$ is saturated. Clearly, increasing the $N_i$'s will not change that we have a witness so without loss of generality $N_{\lambda^+}$ is also saturated. We claim that $M_{\lambda^+} \not \leap{\RWNF} N_{\lambda^+}$. We show this by proving that for any $i < \lambda^+$ and any $j \le i + 1$, $\neg \RWNF (M_j, N_j, M_{i + 2}, N_{i + 2})$. Indeed, $\tp (\ba_{i + 1}, N_j, N_{i + 2})$ forks over $M_j$: if not, then by transitivity $\tp (\ba_{i + 1}, N_j, N_{i + 2})$ does not fork over $M_0$, and hence $\tp (\ba_{i + 1}, N_0, N_{i + 2})$ does not fork over $M_0$, and we know that this is not the case of the witness we selected.

    \item \underline{$\neg (\ref{equiv-2})$ implies $\neg (\ref{equiv-1})$}: Fix $M, N$ saturated in $\K_{\lambda^+}$ such that $M \lea N$ but $M \not \leap{\RWNF} N$. 

      \underline{Claim}: For any $A \subseteq |M|$ of size $\lambda$, there exists $M_0 \leas M_1 \lea M$ and $N_0 \leas N_1 \lea N$ such that $M_0 \leas N_0$, $M_1 \leas N_1$, $A \subseteq |M_0|$, but $\neg \RWNF (M_0, N_0, M_1, N_1)$.

      \underline{Proof of Claim}: If not, we can use failure of the claim and continuity of $\RWNF$ to build increasing continuous resolution $\seq{M_i : i \le \lambda^+}$, $\seq{N_i : i \le \lambda^+}$ of $M$ and $N$ respectively such that $\RWNF (M_i, N_i, M_j, N_j)$ for all $i < j < \lambda^+$. Thus $M \leap{\RWNF} N$, contradicting the assumption. $\dagger_{\text{Claim}}$

      Build $\seq{M_i^\ast : i \le \lambda^+}$, $\seq{N_i^\ast : i \le \lambda^+}$ increasing continuous resolutions of $M$, $N$ respectively such that for all $i < \lambda^+$, $M_i^\ast \leas N_i^\ast$ and $\neg \RWNF (M_{i + 1}^\ast, N_{i + 1}^\ast, M_{i + 2}^\ast, N_{i + 2}^\ast)$. This is possible by the claim. Let $\ba_{i + 1}^\ast \in |M_{i + 2}^\ast|$ witness the $\RWNF$-forking, i.e.\ $\tp (\ba_{i + 1}^\ast, N_{i + 1}^\ast, N_{i + 2}^\ast)$ forks over $M_{i + 1}^\ast$. By Fodor's lemma, local character, and stability, there exists a stationary set $S$, $i_0 < \lambda^+$ and $p \in \oSsbs (M_{i_0}^\ast)$ such that for all $i \in S$, $\tp (\ba_{i + 1}^\ast, M_{i}^\ast, M_{i + 2}^\ast)$ is the nonforking extension of $p$. Without loss of generality, $i_0$ is limit and all elements of $S$ are also limit ordinals.

      Now build an increasing continuous sequence of ordinals $\seq{j_i : i < \lambda^+}$ as follows. Let $j_0 := i_0$. For $i$ limit, let $j_i := \sup_{k < i} j_k$. For $i$ successor, pick any $j_i \in S$ with $j_i > j_{i - 1}$.

      Now for $i$ not the successor of a limit, let $M_i := M_{j_i}^\ast$, $N_i := N_{j_i}^\ast$, $\ba_i := \ba_{j_i}^\ast$. For $i = k + 1$ with $k$ a limit, set $M_i := M_{j_k}^\ast$, $N_i := N_{j_k}^\ast$, $\ba_i := \ba_{j_i}^\ast$. This gives a witness to the failure of being $\goodp$.
    \item \underline{(\ref{equiv-0}) implies (\ref{equiv-2})}: If $\LWNF$ has the symmetry property, then by Remark \ref{trivial-obs}, $\LWNF = \RWNF = \WNF$. By Fact \ref{reflects-lwnf}, it follows that $M \lea N$ implies $M \leap{\WNF} N$ for any $M, N \in \K_{\lambda^+}$, so (\ref{equiv-2}) holds.
  \end{itemize}
\end{proof}

\begin{question}
  Are the conditions in Theorem \ref{sym-equiv} all equivalent?
\end{question}
\begin{question}
  Is there a good $\lambda$-frame $\s$ such that $\LWNF_{\s}$ does \emph{not} have the symmetry property?
\end{question}

The next result shows that the uniqueness property has strong consequences. The first author has given conditions under which when $\lambda = \aleph_0$, failure of uniqueness implies nonstructure \cite[VII.4.16]{shelahaecbook2}.

\begin{thm}\label{uq-consequence}
  Assume that $\s$ is a good $(<\omega, \lambda)$-frame (so it satisfies symmetry). If $\LWNF$ has the uniqueness property, then $\LWNF$ has the symmetry property and $\s$ is successful $\goodp$ (see \cite[III.1.1]{shelahaecbook}).
\end{thm}
\begin{proof}
  By \cite[3.11]{downward-categ-tame-apal} (used with the pre-$(\le \lambda, \lambda)$-frame induced by $\LWNF$, recalling Fact \ref{reflects-lwnf}) $\s$ is weakly successful. This implies that there is a relation $\NF = \NF_{\s}$ that is a nonforking relation respecting $\s$ (see \cite[II.6.1]{shelahaecbook}, in particular it has all the properties listed in Theorem \ref{weak-props}, as well as uniqueness and symmetry). Now as $\NF$ respects $\s$, we must have that $\NF (M_0, M_1, M_2, M_3)$ implies $\LWNF (M_0, M_1, M_2, M_3)$. Since $\LWNF$ has the uniqueness property and $\NF$ has the existence property, it follows from \cite[4.1]{bgkv-apal} that $\LWNF = \NF$. In particular, $\LWNF$ has the symmetry property.

  To see that $\s$ is successful $\goodp$, it is enough to show that for $M, N \in \K_{\lambda^+}$, $M \lea N$ implies $M \leap{\NF} N$ (where $\leap{\NF}$ is defined as in Definition \ref{leanf-def}). This is immediate from Fact \ref{reflects-lwnf} and $\LWNF = \NF$.
\end{proof}

To prepare for the proof of symmetry in the $\lambda = \aleph_0$ case, we end this section by introducing yet another notion of nonforking amalgamation ($\VWNF$ stands for ``very weak nonforking amalgamation''). In this case, we look at finite sequences both on the left \emph{and} the right hand side. We show that if $\s$ is a good frame, then $\VWNF$ has the symmetry property and locality of types implies that $\VWNF = \LWNF$. Thus in this case $\LWNF$ has the symmetry property too.

\begin{defin} \
  \begin{enumerate}
  \item For $M \leas N$, $B \subseteq |N|$, $\ba \in \fct{<\omega}{N}$, we say that $\tp (\ba, B, N)$ \emph{does not fork over $M$} if there exist $M'$, $N'$ with $N \leas N'$, $M \leas M' \leas N'$, and $B \subseteq |M'|$ such that $\tp (\ba, M', N')$ does not fork over $M_0$.
  \item We define a 4-ary relation $\VWNF_{\s} = \VWNF$ on $\K_{\s}$ by $\VWNF (M_0, M_1, M_2, M_3)$ if and only if $M_0 \leas M_\ell \leas M_3$, $\ell = 1,2$ and for any $\ba \in \fct{<\omega}{M_1}$ and any finite $B \subseteq |M_2|$, if $\tp (\ba, M_0, M_3)$ and $\tp (\ba, M_2, M_3)$ are both basic, then $\tp (\ba, B, M_3)$ does not fork over $M_0$.
  \end{enumerate}
\end{defin}

\begin{thm}\label{sym-tameness} \
  Assume that $\s$ is a type-full good $(<\omega, \lambda)$-frame.
  \begin{enumerate}
  \item $\VWNF$ has the symmetry property: $\VWNF (M_0, M_1, M_2, M_3)$ if and only if $\VWNF (M_0, M_2, M_1, M_3)$.
  \item If for any $M \in \K_{\s}$ and any $p \neq q \in \oSsbs (M)$ there exists $B \subseteq |M|$ finite such that $p \rest B \neq q \rest B$, then $\VWNF = \LWNF$. In particular, $\LWNF$ has the symmetry property.
  \end{enumerate}
\end{thm}
\begin{proof} \
  \begin{enumerate}
  \item By the symmetry axiom of good frames.
  \item This is observed in \cite[4.5]{indep-aec-apal}. In details, it suffices to show that for $M \leas N$, $p \in \oSsbs (N)$ does not fork over $M$ if and only if $p \rest B$ does not fork over $M$ for all finite $B \subseteq |N|$. Let $q \in \oSsbs (N)$ be the nonforking extension of $p \rest M$. For any finite $B \subseteq |N|$, we have that $q \rest B = p \rest B$, by the uniqueness property for (the extended notion of) forking, see \cite[5.4]{bgkv-apal}. Therefore by the assumption we must have $p = q$, as desired.
  \end{enumerate}
\end{proof}

\section{Building a superlimit}\label{sl-sec}

In this section, we work in $\aleph_0$ and show assuming $\aleph_0$-stability and amalgamation that $\K$ is $\PC_{\aleph_0}$ (Theorem \ref{def-thm}) and has a superlimit (Corollary \ref{sl-cor}).

\begin{hypothesis} 
$\K = (K, \lea)$ is an AEC with $\LS (\K) = \aleph_0$ (and countable vocabulary).
\end{hypothesis}

We will use without comments Fact \ref{brimmed-uq} and Remark \ref{brimmed-uq-rmk}. The essence of it is that since $\lambda = \aleph_0$ all brimmed models have the same length, and hence are isomorphic (and the isomorphism fixes any common base they may have).

First note that if $\K$ is stable and has few models, we can say something about its definability:

\begin{thm}\label{def-thm}
  Assume that $\Ii (\K, \aleph_0) \le \aleph_0$.
  
  \begin{enumerate}
  \item The set $\{M \in K_{\aleph_0}: |M| \subseteq \omega\}$ is Borel.
  \item If $\K$ has amalgamation in $\aleph_0$ and is stable in $\aleph_0$, then the set $\{(M,N):M \lea N$ and $|N| \subseteq \omega\}$ is $\mathbf{\Sigma^1_1}$.
  \end{enumerate}

  In particular if $\K$ has amalgamation in $\aleph_0$ and is stable in $\aleph_0$, then $\K$ is a $\PC_{\aleph_0}$-representable AEC.
\end{thm}
\begin{proof} \
  \begin{enumerate}
  \item Fix $M \in \K_{\aleph_0}$. By Scott's isomorphism theorem, there exists a formula $\phi_M$ of $\Ll_{\aleph_1, \aleph_0} (\tau_{\K})$ such that $N \models \phi_M$ if and only if $M \cong N$. Now observe that the set
    
    $$
    \{N : N \text{ is a } \tau_{\K}\text{-structure with }|N| \subseteq \omega \text{ and } N \models \phi_M\}
    $$

    is Borel and use that $\Ii (\K, \aleph_0) \le \aleph_0$.
  \item For $M, N \in \K_{\aleph_0}$ with $M \lea N$, let us say that $N$ is \emph{almost brimmed over $M$} if either $N$ is brimmed over $M$, or $N$ is $\lea$-maximal. Using amalgamation, it is easy to check that if $N, N'$ are both almost brimmed over $M$, then $N \cong_{M} N'$ (as in Fact \ref{brimmed-uq}, recalling that the chains witnessing the brimmedness must have cofinality $\omega$). Moreover there always exists an almost brimmed model over any $M \in \K_{\aleph_0}$.

    Fix $\seq{M_n^\ast : n < \omega}$ such that for any $M \in \K_{\aleph_0}$ there exists $n < \omega$ such that $M \cong M_n^\ast$ (possible as $\Ii (\K,\aleph_0) \le \aleph_0$). For each $n < \omega$, fix $N_n^\ast \in \K_{\aleph_0}$ almost brimmed over $M_n^\ast$. We have:

    \begin{itemize}
      \item[$\circledast_1$] For $M, N \in \K_{\aleph_0}$:

        \begin{enumerate}
        \item There is $n < \omega$ and an isomorphism $f: M_n^\ast \cong M$.
        \item If $N$ is almost brimmed over $M$, then any such $f$ extends to $g: N_n^\ast \cong N$.
        \end{enumerate}
    \end{itemize}

    \begin{itemize}
    \item[$\circledast_2$]
      For $M_1, M_2 \in \K_{\aleph_0}$, $M_1 \lea M_2$ if and only if $M_1 \subseteq M_2$ and for some $n_1, n_2 < \omega$, for some $(N,f_1,f_2)$ we have: $M_1 \subseteq M_2 \subseteq N$
      and for $\ell = 1,2$, $f_\ell$ is an isomorphism from $(M_{n_\ell}^*,N_{n_\ell}^*)$ onto $(M_\ell,N)$.
    \end{itemize}

[Why?  The implication ``if" holds by the coherence axiom of AECs.  The implication
``only if" holds as there is $N \in \K_{\aleph_0}$ which is almost brimmed over
$M_2$ (and so $M_2 \lea N$) hence $N$ is almost brimmed over $M_1$ 
  and use $\circledast_1$ above.]

The result now follows from $\circledast_2$.
  \end{enumerate}

  By \cite[3.3]{baldwin-larson-iterated}, it follows that $\K$ is $\PC_{\aleph_0}$.
\end{proof}

We now study homogeneous models and show that they coincide with brimmed models. Note that the homogeneity here is with respect to a set $D$ of \emph{orbital} types.

\begin{defin}
  Let $D$ be a set of orbital types over the empty set and let $M \in \K$. We say that $M$ is \emph{$(D, \aleph_0)$-homogeneous} if it realizes all the types in $D$ and whenever $p \in D$ is the type of an $(n + m)$-elements sequence and $\ba \in \fct{n} M$ realizes $p^n$ (the restriction of $p$ to its first $n$ ``variables''), there exists a sequence $\bb \in \fct{m}{M}$ such that $\ba \bb$ realizes $p$.
\end{defin}

The next result characterizes the countable brimmed model in AECs that are nicely stable in $\aleph_0$ (recall Definition \ref{nicely-stable-def}).

\begin{thm}\label{brimmed-homog-equiv}
  Assume that $\K$ is nicely stable in $\aleph_0$ and let $\K^\ast$ be the class of amalgamation bases in $\K_{\aleph_0}$. Let $M \in \K_{\aleph_0}$. The following are equivalent:

  \begin{enumerate}
  \item $M$ is brimmed.
  \item $M$ is $(\oSp{\K^\ast}^{<\omega} (\emptyset), \aleph_0)$-homogeneous (see Definition \ref{7r.13A}(\ref{13A-4})).
  \end{enumerate}
\end{thm}

\begin{proof}
  Let $M \in \K_{\aleph_0}$. First we show:

  \underline{Claim}: If $M$ is brimmed, $\ba, \bb \in M$, then $\otp_{\K^\ast} (\ba / \emptyset; M) = \otp_{\K^\ast} (\bb / \emptyset; M)$ if and only if there is an automorphism $f$ of $M$ sending $\ba$ to $\bb$. 

  \underline{Proof of Claim}: The right to left direction is clear. Now assume that $\otp_{\K^\ast} (\ba / \emptyset; M) = \otp_{\K^\ast} (\bb / \emptyset; M)$. Say $\seq{M_i : i < \omega}$ witness that $M$ is brimmed, and without loss of generality assume that $\ba, \bb \in M_0$. Then $\otp_{\K^\ast} (\ba / \emptyset; M_0) = \otp_{\K^\ast} (\bb / \emptyset; M_0)$. Since $\K^\ast$ has amalgamation, there exists $M_0' \in \K^\ast$ with $M_0 \lea M_0'$ and $f: M_0 \rightarrow M_0'$ so that $f (\ba) = \bb$. Since $M_1$ is universal over $M_0$, we can assume without loss of generality that $M_0' = M_1$. Now extend $f$ to an automorphism of $M$ using a back and forth argument.

  From the claim, it follows directly that if $M$ is brimmed, then it is $(\oSp{\K^\ast}^{<\omega} (\emptyset), \aleph_0)$-homogeneous. Conversely, the countable $(\oSp{\K^\ast}^{<\omega} (\emptyset), \aleph_0)$-homogeneous model is unique (standard back and forth argument) and so it must also be brimmed.
\end{proof}
\begin{remark}
  By adding constants to the language, we can also characterize brimmed models over $M_0$ as those that are homogeneous for orbital types of finite sequences over $M_0$.
\end{remark}

\begin{cor}\label{sl-cor}
  If $\K$ is nicely stable in $\aleph_0$, then there is a superlimit model of cardinality $\aleph_0$.
\end{cor}
\begin{proof}
  Let $M \in \K_{\aleph_0}$ be brimmed (it exists by nice stability in $\aleph_0$). We claim that $M$ is superlimit. To see this, we check the conditions of Definition \ref{sl-def}. On general grounds, brimmed models are universal in $\K_{\aleph_0}$, are not maximal (from the definition of nice stability), and there is a unique brimmed model of cardinality $\aleph_0$. Still, it is not obvious that if $\seq{M_i : i < \delta}$ is an increasing chain of brimmed models in $\K_{\aleph_0}$ and $\delta < \omega_1$, then $\bigcup_{i < \delta} M_i$ is brimmed. To see this, we use Theorem \ref{brimmed-homog-equiv}: each $M_i$ is ($(\oSp{\K^\ast}^{<\omega} (\emptyset), \aleph_0)$-homogeneous, and it is clear from the definition that an increasing union of such homogeneous models is homogeneous. Thus $M_\delta$ is $(\oSp{\K^\ast}^{<\omega} (\emptyset), \aleph_0)$-homogeneous. By Theorem \ref{brimmed-homog-equiv} again, $M_\delta$ is brimmed, as desired.
 \end{proof}

We have justified assuming amalgamation in the following sense:

\begin{cor}\label{ap-just}
  If $\K$ is nicely stable in $\aleph_0$, then there exists an AEC $\K' = (K', \leap{\K'})$ such that:

  \begin{enumerate}
  \item\label{p1} $\LS (\K') = \aleph_0$.
  \item $\K'_{<\aleph_0} = \emptyset$.
  \item $\tau_{\K'} = \tau_{\K}$.
  \item $K' \subseteq K$ and for $M, N \in K'$, $M \leap{\K'} N$ if and only if $M \lea N$.
  \item\label{p5} For any $M \in \K_{\aleph_0}$ there exists $M' \in \K_{\aleph_0}'$ with $M \lea M'$.
  \item\label{p6} $\K'$ is categorical in $\aleph_0$.
  \item\label{p7} $\K'$ is very nicely stable in $\aleph_0$. In particular it has amalgamation in $\aleph_0$.
  \item\label{p9} $\K'$ is $\PC_{\aleph_0}$.
  \end{enumerate}
\end{cor}
\begin{proof}
  Let $M \in \K_{\aleph_0}$ be superlimit (exists by Corollary \ref{sl-cor}). Let $K_{\aleph_0}' := \{N \in K : N \cong M\}$. Now let $\K'$ be the AEC generated by $(K_{\aleph_0}', \lea)$ (in the sense of \cite[II.23]{shelahaecbook}). One can easily check that $\K'$ is nicely stable in $\aleph_0$; from this and $\aleph_0$-categoricity we get amalgamation in $\aleph_0$, hence (\ref{p7}) holds. As for (\ref{p9}), it follows from Theorem \ref{def-thm}.
\end{proof}

\section{Building a good $\aleph_0$-frame}\label{building-aleph0-frame}

The aim of this section is to build a good $\aleph_0$-frame from nice $\aleph_0$-stability. By Corollary \ref{ap-just}, we may restrict the class to a superlimit so that it is categorical in $\aleph_0$. As before, we assume:

\begin{hypothesis} 
$\K = (K, \lea)$ is an AEC with $\LS (\K) = \aleph_0$ (and countable vocabulary).
\end{hypothesis}

The nonforking relation of the frame will be nonsplitting:

\begin{defin}\label{split-def}
  For $M \in \K_{\aleph_0}$ and $A \subseteq |M|$, $p \in \oS^{<\omega} (M)$ \emph{splits over $A$} if whenever $p = \otp (\ba / M; N)$, there exists $\bb_1, \bb_2 \in M$ such that $\otp (\bb_1 / A; M) = \otp (\bb_1 / A; M)$ but $\otp (\ba \bb_1 / A; N) \neq \otp (\ba \bb_2 / A; N)$.
\end{defin}

The following is proven in \cite[I.5.6]{shelahaecbook}. Similar proofs appear in \cite[3.16]{finitary-aec} or \cite[4.2]{quasimin-five}.

\begin{fact}\label{splitting-thm}
  Assume that $\K$ is nicely stable in $\aleph_0$ and categorical in $\aleph_0$. If $M \in \K_{\aleph_0}$ and $p \in \oS^{<\omega} (M)$, then there exists $A \subseteq |M|$ finite such that $p$ does not split over $A$.
\end{fact}

The following result about nonsplitting will also come in handy. It appears in various forms in the literature, see e.g.\ \cite[4.8]{bv-sat-afml}.

\begin{lem}[Weak uniqueness]\label{uq-lem}
  Assume that $\K$ is nicely stable in $\aleph_0$ and categorical in $\aleph_0$. Let $M \lea N$ both be in $\K_{\aleph_0}$, $p,q \in \oS^{<\omega} (N)$. If both $p$ and $q$ do not split over a finite subset of $M$ and $p \rest A = q \rest A$ for all finite $A \subseteq |M|$, then $p \rest B = q \rest B$ for all finite $B \subseteq |N|$.
\end{lem}
\begin{proof}
  Let $N'$ be brimmed over $N$. Let $\bb_1, \bb_2 \in N'$ realize $p$ and $q$ respectively. Fix $A \subseteq |M|$ finite such that $p$ and $q$ do not split over $A$. Let $B \subseteq |N|$ be finite and let $\bb$ be an enumeration of $B$. Since $M$ is brimmed, there exists $\bb' \in M$ such that $\otp (\bb / A; N) = \otp (\bb' / A; N)$. By nonsplitting, $\otp (\bb_\ell \bb/ A; N') = \otp (\bb_\ell \bb' / A; N')$ for $\ell = 1,2$. Now since $\bb' \in M$ we have by assumption that $p \rest A \bb' = q \rest A \bb'$. Therefore $\otp (\bb_1 \bb' / A; N') = \otp (\bb_2 \bb' / A; N')$. Putting these equalities together, $\otp (\bb_1 \bb / A; N') = \otp (\bb_2 \bb / A; N')$, so $p \rest B = q \rest B$, as desired.
\end{proof}

\begin{defin}\label{gs-defin}
  Let $\K$ be nicely stable in $\aleph_0$ and categorical in $\aleph_0$. We define a pre-$(<\omega, \lambda)$-frame $\gs = (\K^{\s}, \nf,  \oSs^{\bs})$ by:

  \begin{enumerate}
  \item $\K^{\s} = \K$.
  \item For $M_0 \lea M \lea N$ all in $\K_{\aleph_0}$, $n < \omega$, $\ba = \seq{a_i :i < n} \in \fct{n}{N}$, $\nfs{M_0}{\ba}{M}{N}$ holds if and only if $a_i \notin M$ for all $i < n$ and there exists a finite $A \subseteq |M_0|$ so that $\otp (\ba / M; N)$ does not split over $A$.
  \item For $M \in \K_{\aleph_0}$, $\oSp{\gs}^{\bs} (M)$ is the set of all types of finite sequences $\seq{a_i : i < n}$ over $M$ such that for all $i < n$, $a_i \notin M$.

  \end{enumerate}
\end{defin}

In order to prove that $\s$ is a good $\aleph_0$-frame, we will make an additional locality hypothesis. See Example \ref{locality-example} and the next section for setups where it holds.

\begin{defin}\label{loc-def}
  $\K$ is \emph{$(<\aleph_0, \aleph_0)$-local} if for any $M \in \K$  $p, q \in \oS^{<\omega} (M)$, $p \rest A = q \rest A$ for all finite $A \subseteq |M|$ implies $p = q$. We say that $\K$ is \emph{weakly $(<\aleph_0, \aleph_0)$-local} if this holds for a superlimit $M$.
\end{defin}
\begin{remark}
  The definition of locality includes types of any finite length, not just of length one. This will be used to prove the symmetry property of $\LWNF_{\s}$, via Theorem \ref{sym-tameness}.
\end{remark}

We now prove, assuming nice stability, categoricity, and locality, that the pre-frame defined above is a good $\aleph_0$-frame.

\begin{thm}\label{good-frame-thm}
  Assume that $\K$ is nicely stable in $\aleph_0$ and categorical in $\aleph_0$. If $\K$ is $(<\aleph_0, \aleph_0)$-local, then $\gs$ (Definition \ref{gs-defin}) is a type-full good $(<\omega, \aleph_0)$-frame. Moreover $\LWNF_{\s}$ has the symmetry property (recall Definitions \ref{wnf-def} and \ref{sym-def}). In particular, $\s$ is $\goodp$.
\end{thm}
\begin{proof}
  Once we have shown that $\s$ is a type-full good frame, the moreover part follows from Theorem \ref{sym-tameness}. The last sentence is by Theorem \ref{sym-equiv}.

  Now except for symmetry, the axioms of good frames are easy to check (see the proof of \cite[II.3.4]{shelahaecbook}). For example:

  \begin{itemize}
  \item Local character: Let $\seq{M_i : i \le \delta}$ be increasing continuous in $\K_{\gs}$. Let $p \in \oSs^{\bs} (M_\delta)$. By Fact \ref{splitting-thm}, there exists a finite $A \subseteq |M_\delta|$ such that $p$ does not split over $A$. Pick $i < \delta$ such that $A \subseteq |M_i|$. Then $p$ does not fork over $M_i$.
  \item Uniqueness: by Lemma \ref{uq-lem} and locality.
  \item Extension: follows on general grounds, see \cite[3.5]{stab-spec-v5}.
  \end{itemize}
  
  Symmetry is the hardest to prove, and is done as in \cite[I.5.30]{shelahaecbook}. We give a full proof for the convenience of the reader.

  Suppose that $\tp (\bb, N_2, N_3)$ does not fork over $N_0$ and let $\bc \in \fct{<\omega}{N_2 \backslash N_1}$. We want to find $N_1, N_3'$ such that $N_0 \leas N_1 \leas N_3'$, $N_3 \leas N_3'$, $\bb \in \fct{<\omega}{N_1}$ and $\tp (\bc, N_1, N_3')$ does not fork over $N_0$. Assume for a contradiction that there is no such $N_1$. Using existence for $\LWNF_{\s}$ (see Theorem \ref{weak-props}), as well as the extension property for nonforking, we can increase $N_2$ and $N_3$ if necessary and find $N_1$ such that $\LWNF_{\s} (N_0, N_1, N_2, N_3)$, $N_\ell$ is brimmed over $N_0$, and $N_3$ is brimmed over $N_\ell$ for $\ell = 1,2$. By assumption, $p := \tp (\bc, N_1, N_3)$ forks over $N_0$.

  \underline{Claim 1}: Let $I$ be the linear order $[0, \infty) \cap \mathbb{Q}$. There exists an increasing chain $\seq{M_s : s \in I}$ such that for any $s < t$ in $I$, $M_s, M_t$ are in $\K_{\aleph_0}$ and $M_t$ is brimmed over $M_s$.

    \underline{Proof of Claim 1}: Let $\phi \in \Ll_{\omega_1, \omega}$ be a Scott sentence for the model in $\K_{\aleph_0}$. Let $\psi \in \Ll_{\omega_1, \omega}$ be a Scott sentence for a pair $M, N \in \K_{\aleph_0}$ such that $N$ is brimmed over $M$. Now let $K^\ast$ be the class of sequences $\seq{M_j : j \in J}$ such that $J$ is a linear order, $M_j \models \phi$ for all $j \in J$, and $(M_j, M_k) \models \psi$ for all $j < k$ in $J$. It is easy to see that $K^\ast$ is axiomatizable by a sentence in $\Ll_{\omega_1, \omega}$. Moreover, for each $\alpha < \omega_1$, there is a sequence $\seq{M_i : i < \alpha}$ in $K^\ast$. By \cite[Theorem 12(i)]{kei71}, this implies that $K^\ast$ contains an $I$-indexed member, as desired. $\dagger_{\text{Claim 1}}$

  Fix $I$, $\seq{M_s : s \in I}$ as in Claim 1. Fix $N_0'$ such that $N_0$ is brimmed over $N_0'$ and $p \rest N_0$ does not fork over $N_0'$.

  For any fixed infinite $J \subseteq I$, write $M_J := \bigcup_{s \in J} M_s$. Assume now that $M_I$ is brimmed over $M_J$. Let $N_0^J := M_J$, $N_1^J := M_I$. Let $N_3^J$ be brimmed over $N_1^J$. By categoricity and uniqueness of brimmed models, there exists $f_0 : N_0' \cong M_0$, $f_0^J : N_0 \cong N_0^J$, $f_1^J : N_1 \cong N_1^J$, and $f_3^J : N_3 \cong N_3^J$ such that $f_0 \subseteq f_0^J \subseteq f_1^J \subseteq f_3^J$. Let $f_2^J := f_3^J \rest N_2$ and let $N_2^J := f_2^J[N_2]$. Note that $\LWNF_{\s} (N_0^J, N_1^J, N_2^J, N_3^J)$ holds.

  Let $p_J := \tp (f_3^J (\bc), f_3^J[N_1], f_3^J[N_3]) = \tp (f_3^J (\bc), M_I, N_3^J)$. Since we are assuming that $\tp (\bc, N_1, N_3)$ forks over $N_0$, we have that $p_J$ forks over $N_0^J$. Moreover $p_J \rest N_0^J$ does not fork over $M_0$.

  \underline{Claim 2}: If $J$ has no last elements, $I \backslash J$ has no first elements, and $t \in I \backslash J$, then $p_J \rest M_t$ forks over $N_0^J$.

  \underline{Proof of Claim 2}: Suppose that $p_J \rest M_t$ does not fork over $N_0^J$. Note that $M_t$ is brimmed over $M_J$. Find $N_1'$ such that $N_0 \leas N_1' \leas N_1$, $N_1'$ is brimmed over $N_1$, and $f_1^J : N_1' \cong M_t$. Let $\bb' \in \fct{<\omega}{N_1'}$ be such that $\tp (\bb', N_0, N_1') = \tp (\bb, N_0, N_1)$. Since $\LWNF_{\s} (N_0, N_1, N_2, N_3)$, we know that $\tp (\bb', N_2, N_3)$ does not fork over $N_0$, hence by uniqueness $\tp (\bb, N_2, N_3) = \tp (\bb', N_2, N_3)$. But we have assumed that $\tp (\bc, N_1', N_3)$ does not fork over $N_0$ and $\bb' \in \fct{<\omega_1}{N_1'}$, hence by a simple renaming we obtain a contradiction to our hypothesis that symmetry failed. $\dagger_{\text{Claim 2}}$
  
  \underline{Claim 3}: If $J_1 \subsetneq J_2$ are both proper initial segments of $I$ with no last elements and $J_2 \backslash J_1$ has no first elements, then $p_{J_1} \neq p_{J_2}$.

  \underline{Proof of Claim 3}: Fix $t \in J_2 \backslash J_1$. By Claim 2, $p_{J_1} \rest M_t$ forks over $N_0^{J_1}$. We claim that $p_{J_2} \rest M_t$ does not fork over $N_0^{J_1}$. Indeed recall that $N_0^{J_2} = M_{J_2}$ and by assumption $p_{J_2} \rest N_0^{J_2}$ does not fork over $M_0$. Therefore by monotonicity also $p_{J_2} \rest M_t$ does not fork over $M_{J_1} = N_0^{J_1}$. $\dagger_{\text{Claim 3}}$

  To finish, observe that there are $2^{\aleph_0}$ cuts of $I$ as in Claim 3. Therefore stability fails, a contradiction.   
\end{proof}

The next corollary does not assume categoricity, but uses amalgamation in $\aleph_0$, rather than just density of amalgamation bases.

\begin{cor}\label{sl-aleph1-cor}
  If $\K$ is \emph{very} nicely stable in $\aleph_0$ and weakly $(<\aleph_0, \aleph_0)$-local, then $\K$ has a superlimit of cardinality $\aleph_1$.
\end{cor}
\begin{proof}
  By Corollary \ref{sl-cor}, $\K$ has a superlimit $N_0$ in $\aleph_0$. Let $\K'$ be the class generated by this superlimit, as described by the proof of Corollary \ref{ap-just}. Then $\K'$ is categorical in $\aleph_0$ and nicely stable in $\aleph_0$, hence we can apply Theorem \ref{good-frame-thm} and get a type-full $\goodp$ $\aleph_0$-frame with underlying class $\K_{\aleph_0}'$. By Theorem \ref{sym-equiv}, $\K'$ has a superlimit model in $\aleph_1$. This is also a superlimit in $\K$: the only nontrivial property to check is universality. Let $M \in \K_{\aleph_1}$. Fix any $M_0 \in \K_{\aleph_0}$ with $M_0 \lea M$. By universality of $N_0$, there exists $f: M_0 \rightarrow N_0$. Now let $N \in \K'$ be superlimit in $\aleph_1$ with $N_0 \lea N$. Using amalgamation (amalgamation in $\aleph_0$ suffices for this, see \cite[I.2.11]{shelahaecbook}), we can find $g: M \rightarrow N$ extending $f$, as needed.
\end{proof}

\section{Locality from supersimplicity}\label{supersimple-sec}

In this section, we give a sufficient condition for locality. As before, we assume:

\begin{hypothesis} 
$\K = (K, \lea)$ is an AEC with $\LS (\K) = \aleph_0$ (and countable vocabulary).
\end{hypothesis}

In the context of a nicely $\aleph_0$-stable AEC, the following definition generalizes that of a supersimple homogeneous model \cite[2.5(iv)]{buechler-lessmann}. The idea is that we want to have a nice notion of nonforking available for all finite sets (not only models). However we do not require that forking over finite sets satisfies any uniqueness requirement. Thus it is not a-priori clear that supersimplicity implies the existence of a good frame (although this will follow from $\aleph_0$-categoricity and Theorems \ref{good-frame-thm}, \ref{tameness-thm}).

Throughout this section expressions such as ``nonforking'' or ``not fork'' will refer to the relation defined in Definition (\ref{supersimple-def})(\ref{ss-inv}) below. We give examples after the definition.

\begin{defin}\label{supersimple-def}
  Assume that $\K$ is nicely $\aleph_0$-stable and categorical in $\aleph_0$. We say that $\K$ is \emph{supersimple} if there exists a 4-ary relation $\nf$ such that:

  \begin{enumerate}
  \item $\nf (A, B, C, N)$ implies that $N \in \K_{\aleph_0}$, $A \cup B \cup C \subseteq |N|$ and $A, B, C$ are all finite or countable. We write $\nfs{A}{B}{C}{N}$ instead of $\nf (A, B, C, N)$. Below, we may abuse notation and write e.g.\ $A_1A_2$ instead of $A_1 \cup A_2$, or $\nfs{\ba}{\bb}{\bc}{N}$ instead of $\nfs{A}{B}{C}{N}$, where $A$, $B$, $C$ stand for the ranges of $\ba$, $\bb$, and $\bc$ respectively.
  \item Normality: $\nfs{A}{B}{C}{N}$ if and only if $\nfs{A}{AB}{AC}{N}$.
  \item\label{ss-inv} Invariance under $\K$-embeddings: If $f: N \rightarrow N'$ and $A \cup B \cup C \subseteq |N|$, then $\nfs{A}{B}{C}{N}$, if and only if $\nfs{f[A]}{f[B]}{f[C]}{N'}$. This shows that $\nf$ is really a relation on types, so we say \emph{$\otp (\bb / C; N)$ does not fork over $A$} if $\nfs{A}{\bb}{C}{N}$.
  \item Monotonicity: If $\nfs{A}{B}{C}{N}$ and $A \subseteq A' \subseteq B' \subseteq B$ then $\nfs{A'}{B'}{C}{N}$. 
  \item Symmetry: If $\nfs{A}{B}{C}{N}$, then $\nfs{A}{C}{B}{N}$.
  \item Local character: If $M \in \K_{\aleph_0}$ and $p \in \oS^{<\omega} (M)$, then there exists $A \subseteq |M|$ finite such that $p$ does not fork over $A$.
  \item Extension: If $p \in \oS^{<\omega} (C; N)$ does not fork over $A \subseteq C$, then there is $q \in \oS^{<\omega} (N)$ such that $q$ extends $p$ and $q$ does not fork over $A$.
  \item Transitivity: If $\nfs{A}{B}{C}{N}$ and $\nfs{C}{B}{D}{N}$ with $A \subseteq C \subseteq D$, then $\nfs{A}{B}{D}{N}$.
  \item Relationship with splitting: If $M \lea N$ are both in $\K_{\aleph_0}$, $p \in \oS^{<\omega} (N)$, and $p$ does not fork over $M$, then there is $A \subseteq |M|$ finite such that $p$ does not split over $A$.
  \end{enumerate}
\end{defin}
\begin{remark}
  It may be helpful to compare Definition \ref{supersimple-def} with Definition \ref{good-frame-def}. The idea of \ref{supersimple-def} is to give a sort of analog of good frames but to allow types over sets. Note that the statement of symmetry in \ref{good-frame-def} is more technical, precisely because types over sets are not allowed. However the idea is the same. Another difference is that local character in \ref{supersimple-def} is stated as ``every type does not fork over a finite set''. In \ref{good-frame-def}, it is stated as ``every type over the union of an increasing chain does not fork over a previous element of the chain''. Again, the lack of types over sets makes it impossible to state the former in good frames.
\end{remark}
\begin{example} \
  \begin{enumerate}
  \item Working inside a supersimple homogeneous model $N$ (in the sense of \cite[2.5(iv)]{buechler-lessmann}), we can define $\nfs{A}{B}{C}{N}$ to hold if and only if $B$ is $\aleph_0$-free from $C$ over $A$ (in the sense of \cite[2.2]{buechler-lessmann}). The first four conditions of Definition \ref{supersimple-def} are then easy to check. Symmetry is \cite[2.14]{buechler-lessmann} and transitivity is \cite[2.15]{buechler-lessmann}. Local character and extension are given by the definition in \cite[2.5]{buechler-lessmann}. Now if in addition $N$ is $\aleph_0$-stable (in the sense of \cite[5.1]{buechler-lessmann} or equivalently in the sense given here), then by \cite[5.2]{buechler-lessmann} the last axiom of Definition \ref{supersimple-def} holds.
  \item Let $\K$ be a FUR class \cite[2.17]{group-config-kangas-apal} (this includes in particular all quasiminimal pregeometry classes). Then letting $\nf$ be defined as in \cite[2.38]{group-config-kangas-apal}, we can also check that it satisfies Definition \ref{supersimple-def}. 
  \end{enumerate}
\end{example}

We first show that if a type $p$ over a finite set does not fork over a subset $A$ of a countable model $M$, then the type is realized inside $M$.

\begin{lem}\label{nf-realize}
  Assume that $\K$ is nicely $\aleph_0$-stable, supersimple, and categorical in $\aleph_0$. Let $M \lea N$ both be in $\K_{\aleph_0}$, and let $B \subseteq |N|$ be finite. Let $p \in \oS^{<\omega} (B; N)$. If $p$ does not fork over $B \cap |M|$, then $p$ is realized in $M$.
\end{lem}
\begin{proof}
  Extending $N$ if necessary, we can assume without loss of generality that $N$ is brimmed over $M$. Let $q \in \oS^{<\omega} (N)$ be a nonforking extension of $p$. Let $\bb$ be an enumeration of $B$. Fix $A' \subseteq |M|$ finite and big-enough such that $B \cap |M| \subseteq A'$, $\otp (\bb / M; N)$ does not split and does not fork over $A'$ and $q$ does not split over $A'$ (recall Fact \ref{splitting-thm}). Let $M_0 \lea M$ be such that $M$ is brimmed over $M_0$ and $M_0$ contains $A'$. Let $\bd \in M$ be such that $\bd$ realizes $q \rest M_0$. We have that $\nfs{M_0}{\bb}{M}{N}$ so by symmetry and monotonicity, $\nfs{M_0}{\bd}{\bb}{N}$. By extension, we can pick $M_0'$ such that $M_0 \lea M_0' \lea N$, $M_0'$ contains $\bb$, and $\nfs{M_0}{\bd}{M_0'}{N}$. Now as $q \rest M_0'$ does not split over a finite subset of $M_0$ and $\bd$ realizes $q \rest M_0$, we must have by Lemma \ref{uq-lem} that $\bd$ realizes $q \rest B = p$, as desired.
\end{proof}

We will prove locality in supersimple $\aleph_0$-stable AECs by a back and forth argument. More precisely, we start with $M \lea N$, $N$ brimmed over $M$, and elements $\ba_1, \ba_2 \in N$ whose types over every finite subset of $M$ match. First, we will do a back and forth argument to find an automorphism of $N$ sending $\ba_1$ to $\ba_2$ and fixing $M$ \emph{setwise} (Lemma \ref{aut-nf}). We will then use this automorphism and nonsplitting to build another automorphism that fixes $M$ \emph{pointwise} (Theorem \ref{tameness-thm}).

The next lemma starts setting up the stage by making sure that we can map an element of $M$ to an element of $M$.

\begin{lem}\label{ind-m-lem}
  Assume that $\K$ is nicely $\aleph_0$-stable, supersimple, and categorical in $\aleph_0$. Let $M \lea N$ both be in $\K_{\aleph_0}$. Let $\bb_1, \bb_2 \in N$, $\ba_1, \ba_2, \bc_1 \in M$ be such that $\otp (\ba_1 \bb_1; N) = \otp (\ba_2 \bb_2; N)$. If $\nfs{\ba_1}{\bb_1}{\bc_1}{N}$, then there exists $\bc_2 \in M$ such that $\otp (\ba_1 \bb_1 \bc_1; N) = \otp (\ba_2 \bb_2 \bc_2; N)$.
\end{lem}
\begin{proof}
  Extending $N$ if necessary, we can assume without loss of generality that $N$ is brimmed over $M$. By symmetry, $\nfs{\ba_1}{\bc_1}{\bb_1}{N}$. Let $f$ be an automorphism of $N$ sending $\ba_1 \bb_1$ to $\ba_2 \bb_2$. By invariance, $\nfs{\ba_2}{f(\bc_1)}{\bb_2}{N}$ (but we do \emph{not} know that $f(\bc_1) \in M$). Let $q := \otp (f (\bc_1) / \ba_2 \bb_2; N)$. We have to show that $q$ is realized in $M$. Since $q$ does not fork over $\ba_2 \in M$, this is exactly what Lemma \ref{nf-realize} tells us.
\end{proof}

Our main lemma in the back and forth argument will be:

\begin{lem}[The back and forth lemma]\label{b-f-lem}
  Assume that $\K$ is nicely $\aleph_0$-stable, supersimple, and categorical in $\aleph_0$. Let $M \lea N$ both be in $\K_{\aleph_0}$ with $N$ brimmed over $M$. Let $\bb_1, \bb_2 \in N$, $\ba_1, \ba_2 \in M$ be such that $\otp (\ba_1 \bb_1; N) = \otp (\ba_2 \bb_2; N)$ and $\nfs{\ba_\ell}{\bb_\ell}{M}{N}$ for $\ell = 1,2$. If $\bd_1 \in N$, $\bd_1' \in M$, there exists $\bc_1, \bc_2, \bd_2' \in M$ and $\bd_2 \in N$ such that:
 
  \begin{enumerate}
  \item $\otp (\ba_1 \bb_1 \bc_1 \bd_1' \bd_1; N) = \otp (\ba_2 \bb_2 \bc_2 \bd_2' \bd_2; N)$.
  \item $\nfs{\ba_\ell \bc_\ell \bd_\ell'}{\bb_\ell \bd_\ell}{M}{N}$ for $\ell = 1,2$.
  \end{enumerate}
\end{lem}
\begin{proof}
  Using Lemma \ref{ind-m-lem}, we can enlarge $\ba_1$ and $\ba_2$ if necessary to assume without loss of generality that $\bd_1'$ is empty. Now using local character, fix $\bc_1 \in M$ such that $\nfs{\ba_1 \bc_1}{\bb_1 \bd_1}{M}{N}$. By Lemma \ref{ind-m-lem}, there exists $\bc_2 \in M$ such that $\otp (\ba_1 \bb_1 \bc_1; N) = \otp (\ba_2 \bb_2 \bc_2; N)$. Let $f$ be an automorphism of $N$ witnessing this. By extension and monotonicity, pick $\bd_2 \in N$ such that $\otp (\bd_2 / \ba_2 \bb_2 \bc_2; N) = \otp (f (\bd_1) / \ba_2 \bb_2 \bc_2; N)$ and $\nfs{\ba_2 \bb_2 \bc_2}{\bd_2}{M \bb_2}{N}$. It remains to see that $\nfs{\ba_2 \bc_2}{\bb_2 \bd_2}{M}{N}$. We do this using a standard nonforking calculus argument: by normality, $\nfs{\ba_2 \bb_2 \bc_2}{\bb_2 \bd_2}{M \bb_2}{N}$ and we also know from the hypotheses of the lemma that $\nfs{\ba_2}{\bb_2}{M}{N}$, so by monotonicity and normality $\nfs{\ba_2 \bc_2}{\ba_2 \bb_2 \bc_2}{M}{N}$. Now using transitivity, monotonicity, and symmetry, $\nfs{\ba_2\bc_2}{\bb_2 \bd_2}{M}{N}$, as desired.
\end{proof}

We can now build the desired automorphism which fixes $M$ setwise.

\begin{lem}\label{aut-nf}
  Assume that $\K$ is nicely $\aleph_0$-stable, supersimple, and categorical in $\aleph_0$. Let $M \in \K_{\aleph_0}$, $A \subseteq |M|$ finite, $p, q \in \oS^{<\omega} (M)$ such that $p$ and $q$ do not fork over $A$. If $p \rest A = q \rest A$, then there exists an automorphism $f$ of $M$ fixing $A$ such that $f (p) = q$.
\end{lem}
\begin{proof}
  Let $N \in \K_{\aleph_0}$ be brimmed over $M$. Say $p = \otp (\bb_1 / M; N)$, $q = \otp (\bb_2 / M; N)$. Let $\ba$ be an enumeration of $A$ and let $\ba_\ell := \ba$, $\ell = 1,2$. Now apply Lemma \ref{b-f-lem} repeatedly in a back and forth argument to build an automorphism $g$ of $N$ fixing $A$ such that $g (\bb_1) = \bb_2$ and $g[M] = M$. Let $f := g \rest M$.
\end{proof}

We now show that we can actually build an automorphism fixing $M$ \emph{pointwise}. The next lemma is the main argument for this:

\begin{lem}\label{finite-eq-lem}
  Assume that $\K$ is nicely $\aleph_0$-stable, supersimple, and categorical in $\aleph_0$. Let $M \in \K_{\aleph_0}$ and let $N$ be brimmed over $M$. Let $\bb_1, \bb_2 \in N$ be such that $\otp (\bb_1 / A; N) = \otp (\bb_2 / A; N)$ for all finite $A \subseteq |M|$. For any $\bd_1 \in N$, there exists $\bd_2 \in N$ such that $\otp (\bb_1 \bd_1 / A; N) = \otp (\bb_2 \bd_2 / A; N)$ for all finite $A \subseteq |M|$.
\end{lem}
\begin{proof}
  Let $p := \otp (\bb_1 / M; N)$, $q := \otp (\bb_2 / M; N)$. Fix $A \subseteq |M|$ finite such that both $p$ and $q$ do not fork over $A$ and $\otp (\bb_1 \bd_1 /M ; N)$ does not split over $A$. By Lemma \ref{aut-nf}, there exists an automorphism $f$ of $M$ fixing $A$ such that $f (p) = q$. Let $g$ be an automorphism of $N$ extending $f$ such that $g (\bb_1) = \bb_2$ and let $\bd_2 := g (\bd_1)$. We claim that this works. Fix $C \subseteq |M|$ finite and let $\bc_2$ be an enumeration of $C$. Let $\bc_1 := g^{-1} (\bc_2)$. We know that $\otp (\bc_1 / A; N) = \otp (\bc_2 / A; N)$, so by nonsplitting, $\otp (\bb_1 \bd_1 \bc_1 / A; N) = \otp (\bb_1 \bd_1 \bc_2 / A; N)$. Applying $g$, we have that $\otp (\bb_1 \bd_1 \bc_1 / A; N) = \otp (\bb_2 \bd_2 \bc_2 / A; N)$. Putting the two equalities together, $\otp (\bb_1 \bd_1 \bc_2 / A; N) = \otp (\bb_2 \bd_2 \bc_2 / A; N)$, so $\otp (\bb_1 \bd_1 / C; N) = \otp (\bb_2 \bd_2 / C; N)$, as desired.
\end{proof}

We have arrived to the main theorem of this section:

\begin{thm}\label{tameness-thm}
  If $\K$ is nicely $\aleph_0$-stable, supersimple, and categorical in $\aleph_0$, then $\K$ is $(<\aleph_0, \aleph_0)$-local (recall Definition \ref{loc-def}).
\end{thm}
\begin{proof}
  Let $M \in \K_{\aleph_0}$ and let $p, q \in \oS^{<\omega} (M)$ be such that $p \rest A = q \rest A$ for all finite $A \subseteq |M|$. Let $N$ be brimmed over $M$ and let $\bb_1, \bb_2 \in N$ realize $p$ and $q$ respectively. Now apply Lemma \ref{finite-eq-lem} in a back and forth argument to get an automorphism of $N$ fixing $M$ taking $\bb_1$ to $\bb_2$.
\end{proof}

\bibliographystyle{amsalpha}
\bibliography{F709}

\end{document}